\documentclass[thmsa,11pt]{article}
\usepackage{authblk}
\normalfont\sffamily
\usepackage{amsfonts}
\usepackage{amssymb}
\usepackage[active]{srcltx}
\usepackage{color}
\usepackage[x11names]{xcolor}
\usepackage{amsmath}
\usepackage{mathtools}
\usepackage{amsthm}
\usepackage{bm} 
\usepackage{graphicx}
\usepackage{textcomp}
\usepackage{a4wide}
\usepackage{multirow}
\usepackage{diagbox}
\usepackage{color}
\usepackage{makecell}
\usepackage{longtable}
\usepackage{cellspace}
\usepackage{tabularx}
\usepackage{array,booktabs,ragged2e}
\usepackage{pxfonts}
\usepackage{marvosym}
\usepackage[format=hang]{caption}
\usepackage{subcaption}
\usepackage{url}
\usepackage{xfrac}

\usepackage[normalem]{ulem}
\usepackage{comment}
\usepackage{pxfonts}
\usepackage[margin=3cm]{geometry} 
\usepackage{titlesec}

\usepackage{setspace}
\usepackage{footmisc}
\usepackage{appendix}

\usepackage{array,booktabs,ragged2e}
\newcolumntype{R}[1]{>{\RaggedRight}p{#1}}

\usepackage{cellspace}
\usepackage{hyperref}
\newtheorem{theorem}{Theorem}
\newtheorem{lemma}{Lemma}

\newtheorem{proposition}{Proposition}
\newtheorem{corollary}{Corollary}
\newtheorem{definition}{Definition}
\newtheorem{conjecture}{Conjecture}

\def \be {\begin{equation}}
\def \ee {\end{equation}}

\renewcommand{\thesubsection}{\arabic{subsection}}
\makeatletter
\def\@seccntformat#1{\@ifundefined{#1@cntformat}%
   {\csname the#1\endcsname\quad}
   {\csname #1@cntformat\endcsname}}
\newcommand\section@cntformat{}     
\makeatother

\title{Squares with three digits}
\author[2]{Michael Gei\ss er}
\author[2]{Theresa K\"orner}
\author[1]{Sascha Kurz}
\author[2]{Anne Zahn}
\affil[1]{University of Bayreuth, sascha.kurz@uni-bayreuth.de}
\affil[2]{University of Bayreuth, students participating in the course {\lq\lq}Integer Sequences{\rq\rq} in summer term 2021, \texttt{firstname}.\texttt{lastname}@uni-bayreuth.de}
\date{}

\begin{document}
\setstretch{1.2}


\maketitle

\subsection{Introduction}

Consider the equations 
{\footnotesize$$
  2108436491907081488939581538^2 = 4445504440405440505004450045555054500055550554550445444
$$
}
and
{\footnotesize$$
  10100000000010401000000000101^2 = 102010000000210100200000110221001000002101002000000010201.
$$
}
\noindent
\!The common property of the numbers on the left hand sides is that their squares consist of three different digits only. The origin of the question of squares 
with few different digits dates back at least to $1991$, see \cite[Exercise 2.2]{vardi1991computational}  where squares consisting entirely of the three digits $2$, 
$4$, and $8$ are considered. In \cite[Problem F24]{guy1994unsolved} the question of Sin Hitotumatu whether there are only finitely many squares, not ending with zero, 
with just two different decimal digits, is reported. Today only $24$ such numbers are known, the largest being $81619^2=6661661161$. Since there is not so much to 
discover for two digits, people moved over to the study of squares with three different digits. A corresponding webpage, started in 1996 and still updated, is 
\url{www.worldofnumbers.com/threedigits.htm}. And indeed, there a still things to discover. The two examples, mentioned in the first few lines, were unknown so far and the 
second can even be generalized to a $2$-parameter family yielding an infinite number of examples.

Of course, we should not let ourselves get distracted by our fingers and toes. So, given an integer $B\ge 2$, called \emph{base} or \emph{radix}, each positive integer $N$ can 
be uniquely written as
$$
  N=\sum_{i=0}^{n-1} a_iB^i,
$$
where $a_i\in\{0,1,\dots,B-1\}$ are called \emph{digits} for all $0\le i\le n-1$, $n\in\mathbb{N}_{>0}$ is the \emph{length}, which is determined via the condition $a_{n-1}\neq 0$. 
The \emph{base-$B$-representation} on $N$ is denoted by
$$
  \varphi_B(N):=\left[a_{n-1},a_{n-2},\dots,a_1,a_0\right]_B,
$$
where we sometimes drop the commas and the square brackets with index $B$ whenever the base is clear from the context. The digit $a_0$ is also called the \emph{last digit} and 
$l_B(N):=n$ the \emph{length} of $N$ (in base $B$).

Here we are interested in positive integers $N$ such that the base-$B$-representation of $N^2$ contains exactly three different digits, i.e., if $N^2=\left[b_{n-1},\dots,b_0\right]_B$, 
then we require $|\{b_i\,:\, 0\le i\le n-1\}|=3$. As an abbreviation we set $D_B(N):=\{a_i\,:\,0\le i\le n-1\}$, where $N=\left[a_{n-1},\dots,a_0\right]_B$. An example 
in base $B=10$ is given by $N=130834904430015239$, $N^2=17117772217211221211117217772227121$, and $D_{10}(N^2)=\{1,2,7\}$. If $0\in D_B(N^2)$, then $D_B(\left(B^m\cdot N\right)^2)=D_B(N^2)$ 
for each positive integer $N$ and each non-negative integer $m$. So, we restrict ourselves to the cases where $N$ is not divisible by $B$.
 
In principle, for $B\ge 4$ positive integers $N$ with $|D_B(N^2)|=3$ are rather rare. However, several {\lq\lq}infinite patterns{\rq\rq} exist. An example for base $B=10$ is given by
\begin{equation}
  \label{ex_infinite_pattern}
  [1\,\underset{k\text{ times}}{\underbrace{6\dots 6}}\,5]^2 
  =[2\,\underset{k\text{ times}}{\underbrace{7\dots 7}}\,\underset{k+1\text{ times}}{\underbrace{2\dots 2}}\,5],
\end{equation} 
where $k\ge 0$. 
 
The aim of this paper is to summarize the current knowledge on squares with three digits, scattered around webpages and newsgroup posting, and to add a few new insights. While we 
will mostly focus on the base $B=10$, several results are presented for general values of $B$. The used mathematical tools are completely elementary. However, we give complete proofs 
of all statements or explicitly state them as conjectures. The closest {\lq\lq}research paper{\rq\rq}, that we are aware of is \cite{bennett2017squares}. Here the authors determine 
all integers $N$ such that $N^2$ has at most three non-zero digits in base $q\in \{2,3,4,5,8,16\}$. For odd primes $q$, $u>v>3$, and $r^2,s,t<q$ they also determine all solutions of the 
Diophantine equation
$$
  N^2 =r^2 +s\cdot q^u+t\cdot q^v.
$$ 
We are not aware of any number theoretic tools that might be applied if the non-zero digits are not sparse but they attain few distinct values only.

The remaining part of this paper is structured as follows. In Section~\ref{sec_computing_digits} we briefly consider the computation of the digits of a square. 
Infinite patters of integers whose squares consist of just three decimal digits are the topic of Section~\ref{sec_infinite_patterns}. Feasible and infeasible 
triples of digits are discussed in Section~\ref{sec_triples}. We consider the properties of admissible suffixes and prefixes in Section~\ref{sec_suffices} before 
we briefly summarize the known enumeration algorithms in Section~\ref{sec_search_algorithms}. A few preliminary results for base $B=9$ are given in Appendix~\ref{sec_b_9}.
 
\subsection{Computing the sequence of digits of squares}
\label{sec_computing_digits}
By $a\operatorname{rem} b$ we denote the remainder of the division of $a$ by $b$ and by $a\operatorname{div}b$ we 
abbreviate $\left\lfloor a/b\right\rfloor$ whenever $a,b\in\mathbb{N}_{>0}$. With this we have 
$$
  (a\operatorname{div}b) \cdot b+(a\operatorname{rem} b) =a.
$$   

Surely, we all learned in elementary school how to perform long multiplication. So, the following lemma is just a more formal 
statement settling notation for the case of squaring.  
\begin{lemma}
  \label{lemma_mult_last_digits}
  Let $B\ge 2$ be a base and $a_0,a_{n-1}\in\{0,\dots,B-1\}$ be arbitrary digits. If 
  $$
    \left[b_{m-1},\dots,b_0\right]_B=\big(\left[a_{n-1},\dots,a_0\right]_B\big)^2,
  $$
  then we have
  \begin{equation}
    b_j=\left(c_{j-1}+\sum_{i=0}^j a_i\cdot a_{j-i}\right)\operatorname{rem}B
  \end{equation}
  for all $0\le j\le n-1$, where $c_{-1}:=0$ and 
  \begin{equation}
    c_j:=\left(c_{j-1}+\sum_{i=0}^j a_i\cdot a_{j-i}\right)\operatorname{div}B
  \end{equation}
  for all $0\le j\le n-1$. For all $n\le j\le 2n-2$ we have
  \begin{equation}
    b_j=\left(c_{j-1}+\sum_{i=j-n+1}^{n-1} a_i\cdot a_{j-i}\right)\operatorname{rem}B,
  \end{equation}
  where
  \begin{equation}
    c_j=\left(c_{j-1}+\sum_{i=j-n+1}^{n-1} a_i\cdot a_{j-i}\right)\operatorname{div}B.
  \end{equation}
  For $j=2n-1$ we have $b_j=c_{j-1}$ and $c_j=0$.
\end{lemma}
For the 	carryovers $c_j$ we have the following bounds.
\begin{lemma}
  \label{lemma_carryover_bounds}
  With the notation and conditions from Lemma~\ref{lemma_mult_last_digits} we have 
  \begin{equation}
    \label{eq_c_ub_1}
    0 \le c_j \le (j+1)B-(j+2)=(j+1)(B-1)-1
  \end{equation}
  for all $0\le j\le n-1$ and 
  \begin{equation}
    \label{eq_c_ub_2}
    0\le c_{n-1+j}\le (n-j)B-(n-j)=(n-j)(B-1)
   \end{equation}
   for all $1\le j\le n-1$. 
\end{lemma}
\begin{proof}
  For the first part we consider
  $$
    0\le \sum_{i=0}^j a_i\cdot a_{j-i}\le (j+1)(B-1)^2,
  $$
  so that, obviously, $c_j\ge 0$ for all $0\le j\le n-1$ and 
  $$
    c_0=\left\lfloor\left(c_{j-1}+\sum_{i=0}^j a_i\cdot a_{j-i}\right) /B\right\rfloor \le \left\lfloor(B-1)^2 /B\right\rfloor =B-2.
  $$
  Using induction we compute
  $$
    c_j=\left\lfloor\left(c_{j-1}+\sum_{i=0}^j a_i\cdot a_{j-i}\right) /B\right\rfloor \le \left\lfloor \left((j+1)(B-1)^2+jB-j-1\right)  /B\right\rfloor =(j+1)B-(j+2).
  $$
  
  The second part can be treated analogously.
\end{proof}

We remark that $b_{2n-1}\le B-1$ can be concluded from Equation~(\ref{eq_c_ub_2}), reflecting the fact that squaring an $n$-digit number gives a number with at most $2n$ digits.

Assume that you want to compute just the $189$th digit of the square of an $100$-digit number. The previous two lemmas can be used to accomplish this task without 
computing the last $190$ digits. More generally, let us assume that we want to compute $b_j$. To end we start by bounding the carryover $c_{j-\alpha}$, for some suitably 
large value of $\alpha$, by $\underline{c}_{j-\alpha}\le c_{j-\alpha}\le \overline{c}_{j-\alpha}$. Here we can choose $\underline{c}_{j-\alpha}=0$ and take 
$\overline{c}_{j-\alpha}$ according to Lemma~\ref{lemma_carryover_bounds}. Starting from there we can compute the values $\underline{b}_i$, $\underline{c}_i$, 
$\overline{b}_i$, $\overline{c}_i$ for all $j-\alpha+1\le i\le j$ according to the formulas in Lemma~\ref{lemma_mult_last_digits}. We can easily check that
\begin{equation}
  \overline{c}_i-\underline{c}_i \le \left\lceil\frac{\overline{c}_{i-1}-\underline{c}_{i-1}}{B}\right\rceil,   
\end{equation}
i.e., the estimates for the carryovers $c_i$ are getting better and better. Once we reach $\overline{c}_{i^\star}=\underline{c}_{i^\star}$, we have $\underline{b}_i=b_i=\overline{b}_i$ 
for all indices $i>i^\star$. While the differences decrease logarithmically in the first steps, they can nevertheless stay at $1$ for a long time.  As an example consider the computation 
of the digit $b_{24}$ in 
$$
  123123999999999999^2=15159519375999999753752000000000001,
$$
where the root consists of $n=18$ digits. Choosing $\alpha=7$, we set $\underline{c}_{17}=0$ and $\overline{c}_{17}=18\cdot 9-1=161$. With this, we compute:
\begin{center}
\begin{tabular}{rrrrrrrrrrr}
\hline
$n$ & 18 & 19 & 20 & 21 & 22 & 23 & 24 & 25 & 26 & 27 \\
\hline 
$\underline{b}_{n}$ & 1 & 2 & 9 & 9 & 9 & 9 & 5 & 7 & 3 & 9 \\ 
$\overline{b}_{n}$  & 2 & 8 & 0 & 0 & 0 & 0 & 6 & 7 & 3 & 9 \\
\hline
$\underline{c}_{n}$ & 62 & 60 & 51 & 42 & 33 & 24 & 19 & 15 & 13 & 8 \\ 
$\overline{c}_{n}$  & 78 & 61 & 52 & 43 & 34 & 25 & 19 & 15 & 13 & 8 \\
\hline
\end{tabular}
\end{center}
So, even choosing the relatively large value $\alpha=7$, we only get the information $b_{24}\in\{5,6\}$. Of course, the long sequence of digits $B-1$ in the square causes this effect. 
However, all subsequent digits are computed correctly (and a much smaller value of $\alpha$ would be sufficient to this end).

If $n\le j<2n$, then for the computation of $b_j$ choosing $\alpha=\left\lceil 1+\log_B(2n-j+1)\right\rceil$ is sufficient in most cases, so that 
$O((2n-j)\log_B(2n-j+1))$ additions and multiplications for single digits are needed. If $j$ is close to $2n$, then this significantly improves upon the running time 
$O(n\log n\log\log n)$ of the Sch\"onhage-Strassen algorithm for the multiplication of two $n$-digit numbers.  

However, we will use the recursion from Lemma~\ref{lemma_mult_last_digits} mainly in the case when $1\le j\le n-1$. If we assume that $c_{j-1}$ is stored, then we can compute 
$b_j$ and $c_j$ in $O(j)$ additions and multiplications, which saves us factor of $\log j\log \log j$ at the very least. This can be used if we build up a list of $n$-digit numbers 
whose last $n$ digits of the square are contained in some set $M$ from $(n-1)$-digit numbers whose last $n-1$ digits of the square are contained in $M$, see Section~\ref{sec_search_algorithms}.  

\subsection{Infinite patterns}
\label{sec_infinite_patterns} 

In Equation~(\ref{ex_infinite_pattern}) we have given an example of a $1$-parametric family of integers whose square consists of the decimal digits $2$, $5$, and $7$ only. Later on 
we will see examples with more than one parameter. So, by an \emph{infinite pattern} we may understand a function $f\colon\mathbb{N}^k\to\mathbb{N}$ such that $f(x)^2$ consists of 
at most three different digits in a given base $B$ for all $x\in\mathbb{N}^k$ or a suitable infinite subset of $\mathbb{N}^k$. However, we are not interested in functions $f$ repeating 
the very same values again and again. Also defining $f(n)$ as the $n$th number whose square consists of (at most) three different digits is not what we have in mind. So,  
we are not able to give a precise definition of an {\lq\lq}infinite pattern{\rq\rq} and leave its meaning at the intuitive level and state examples of certain more structured types in the 
subsequent subsections. In Subsection~\ref{subsec_infinite_non_zero} we consider integers $N$ consisting of some fix prefix, an increasing chain of digits all equal to a fix number $x\neq 0$, 
and a fix suffix. In Equation~(\ref{ex_infinite_pattern}) the prefix is $1$, the repeating digit is $x=6$, and the suffix is $5$. The special case where the repeating digit equals zero 
is treated in Subsection~\ref{subsec_infinite_zero-blocks}. Here we can also have several different blocks of increasing sequences of zeroes, of the same or different lengths. In the 
latter case we obtain infinite families with more than one parameter. There are some constructions where a certain pattern is repeated by not just copying but in a more intricate way, 
see Subsection~\ref{subsec_intricate_patterns}. In Subsection~\ref{subsec_summary_infinite_patterns} we summarize our findings and refer to the respective general results from the previous 
subsections for all known and newly discovered infinite patterns for base $B=10$ mentioned at \url{http://www.worldofnumbers.com/threedigits.htm}. 

Numbers $N$ such that $N^2$ contains at most three digits in a given base $B$ and that do not belong to one of the (known) infinite patterns are called \emph{sporadic solutions}. 
Of course, all $1\le N< \sqrt{B^3}$ result in numbers that have at most $3$ digits in total. These will also be called \emph{trivial solutions}. For $B\in\{2,3\}$ the square $N^2$ 
of every positive integer $N$ consists of at most three different digits in base $B$ representation. Here we will also speak of trivial solutions and will mainly focus on $B\ge 4$ 
in the following. For sporadic solutions with many digits we refer to Subsection~\ref{subsec_sporadic_solutions}.

\subsubsection{Infinite patterns with sequences of non-zero digits}
\label{subsec_infinite_non_zero}
In order to explain our first type of an infinite pattern we first provide a few auxiliary results.
\begin{lemma}
  \label{lemma_ending_pattern_bbb}
  Let $B\in\mathbb{N}_{\ge 2}$ and $N$ be a positive integer whose last three digits end with the same integer $0\le x\le B-1$ in base $B$ representation, i.e., 
  $a_0=a_1=a_2=x$ using the notation from Lemma~\ref{lemma_mult_last_digits}. If the last three digits of 
  $N^2$ are identical in base $B$, i.e.\ $b_0=b_1=b_2$, then $x=0$.
\end{lemma}   
\begin{proof}
  We uniquely write $x^2=aB+b$ with integers $0\le a,b\le B-1$, so that the last digit of $N^2$ equals $b$ and the carryover is $a$. Next we compute
  $$
    2x^2+a=2aB+(2b+a),  
  $$
  so that using Lemma~\ref{lemma_mult_last_digits} and its notation yields $b_1\equiv 2b+a \pmod B$. From $b_1=b$ and $0\le a+b\le 2B-2$ we then conclude either $a=b=0$, so that 
  $x=0$, or
  \begin{equation} 
    a+b=B.
  \end{equation}    
  In the following we only consider the latter case, where $c_1=2a+1$. With this we compute
  $$
    c_2B+b_2 \overset{!}{=} 3x^2+2a+1=3aB+3b+2a+1=(3a+2)B+(b+1).
  $$
  Since $b_2=b$, this is impossible.
\end{proof}
Note that for $B\ge 5$ it is possible that the last two digits of $N$ in base $B$ are equal to $1\le x\le B-1$ and also the last two digits of $N^2$ in base $B$ coincide. 
A parametric example is given by $x=B-2$ where $N^2$ ends with two $4$s, e.g., for $B=10$ we have $[88]^2=[7744]$. 
\begin{lemma}
  \label{lemma_ending_pattern_bbcc}
  Let $B\in\mathbb{N}_{\ge 2}$ and $N$ be a positive integer whose last four digits end with the same integer $0\le x\le B-1$ in base $B$ representation, i.e., $a_0=a_1
  =a_2=a_3=x$. If the last two digits of $N^2$ are identical in base $B$ and the last but two's digit equals the last but three's digit, i.e.\ $b_0=b_1$ and $b_2=b_3$, then $x=0$.
\end{lemma} 
\begin{proof}
  For $B=2$ either $x=0$ or $x=1$. In the latter case $\big([11]_2\big)^2$ ends with $[01]_2$, so that we can assume $B\ge 3$ in the following. As in 
  the proof of Lemma~\ref{lemma_ending_pattern_bbb} we uniquely write $x^2=aB+b$ and conclude $a+b=B$ if $x\neq 0$. From 
  $$
    c_2B+b_2 \overset{!}{=} 3x^2+2a+1=3aB+3b+2a+1=(3a+2)B+(b+1).
  $$
  we conclude either $b_2=0$, $c_2=3a+3$, $b=B-1$, $a=1$ or $b_2=b+1$, $c_2=3a+2$, $b\le B-2$. In the first case we compute
  $$
    c_3B+b_3 \overset{!}{=} 4x^2+3a+3=4aB+4b+3a+3=(4a+3)B+(b+3)=(4a+4)B+2.
  $$ 
  From $b_3=b_2=0$ we conclude $B=2$, which is a contradiction. Thus we consider the second case, where
  $$
    c_3B+b_3 \overset{!}{=} 4x^2+3a+2=4aB+4b+3a+2=(4a+3)B+(b+2).
  $$
  Since $b_3=b+1$, this yields a contradiction again.
\end{proof}

So, if $\big([\underset{n\text{ times}}{\underbrace{x,\dots,x}}]_B\big)^2$, where $1\le x\le B-1$ and $n\ge 4$ should contain three subsequent identical digits within its last four digits, 
then the pattern is of the form $[cccb]_B$, where $c\neq b$. It is indeed possible for general $n$ that the last $n$ digits except the last one are identical:
\begin{lemma}
  \label{lemma_all_digit_equal_x_squared}
  Let $B\in\mathbb{N}_{\ge 2}$ and $1\le x\le B-1$ be an integer such that $x^2=(B-1-b)\cdot B+b$ for a suitable integer $b$ satisfying $1\le b\le B-1$. Then, we have
  $$
    \big([\underset{n\text{ times}}{\underbrace{x,\dots,x}}]_B\big)^2=[\underset{n-1\text{ times}}{\underbrace{B-b,\dots,B-b}}\,,B-1-b,\, 
    \underset{n-1\text{ times}}{\underbrace{b-1,\dots,b-1}}\,,b ]_B
  $$
  for each $n\in\mathbb{N}_{>0}$.
\end{lemma} 
\begin{proof}
  First we note that $x^2=(B-1-b)\cdot B+b=(B-1)(B-b)$.
  The integer $N$ that is squared on the left hand side is given by
  \begin{eqnarray*}
    N&=&\sum_{i=0}^{n-1} xB^i=x\cdot\frac{B^n-1}{B-1},
  \end{eqnarray*}
  so that
  \begin{eqnarray*}
    N^2 &=& x^2\cdot\left(\frac{B^n-1}{B-1}\right)^2 
    =(B-b)\cdot \frac{B^{2n}-2B^n+1}{B-1}.
  \end{eqnarray*}  
  The integer $M$ representing the ride hand side is given by
  \begin{eqnarray*}
    M&=&b\,+\,\sum_{i=1}^{n-1} (b-1)B^i\,+\, (B-1-b)B^n\,+\,\sum_{i=n+1}^{2n-1} (B-b)B^i \\ 
    &=& b+(b-1)B\cdot\frac{B^{n-1}-1}{B-1} +(B-1-b)B^n+(B-b)B^{n+1}\cdot \frac{B^{n-1}-1}{B-1} \\ 
    &=& b+b\cdot\frac{B^{n}-B}{B-1}-\frac{B^{n}-B}{B-1}-B^n + (B-b)B^n+ (B-b)\cdot \frac{B^{2n}-B^{n+1}}{B-1} \\ 
    &=& b\cdot \frac{B^n-1}{B-1}- B\cdot \frac{B^n-1}{B-1} + (B-b)\cdot \frac{B^{2n}-B^{n}}{B-1} \\ 
    &=& (B-b)\cdot \frac{B^{2n}-2B^n+1}{B-1},
  \end{eqnarray*}
  i.e., $M=N^2$. 
\end{proof}

Next we will see that the condition on the representation of $x^2$ is not at as special as it looks at first sight.
\begin{lemma}
  \label{lemma_all_digit_equal_x_squared_condition}
  Let $B\in\mathbb{N}_{\ge 2}$ and $1\le x\le B-1$ be an integer such that the last $n$ digits in base-$B$-representation of 
  $$
    \big([\underset{n\text{ times}}{\underbrace{x\dots x}}]_B\big)^2
  $$
  are given by
  $$
    [\underset{n-1\text{ times}}{\underbrace{c\dots c}}\,b]_B,
  $$
  where $b,c\in\{0,\dots,B-1\}$ are arbitrary but fixed, for all $n\in\{1,2,3\}$. 
  Then, $B-1$ divides $x^2$, $b=x^2\operatorname{rem} B$, $c=b-1$, $x^2\operatorname{div}B=B-1-b$, and $b\ge 1$.
\end{lemma} 
\begin{proof}
  We set $b:=x^2\operatorname{rem} B$ and $a:=x^2\operatorname{div}B$, so that $a,b\in\{0,\dots,B-1\}$ and
  \begin{equation}
    \label{eq_lip_1}
    x^2 = a\cdot B+b,
  \end{equation}
  i.e., we apply the condition on the shape of the digits for $n=1$. Observe that $a=b=B-1$ is impossible since $x^2=(B-1)B+(B-1)=B^2-1$ has no integral solution $x$ for $B\ge 2$. 
  Thus, we have
  $$
    0\le a+b\le B-3.
  $$ 
  Next we apply the condition for $n=2$ and $n=3$. 
  Using Lemma~\ref{lemma_mult_last_digits}, including the notation, we have $b_0=b$, 
  $b_1=b_2=c$, 
  $c_0=a$ and we conclude
  \begin{eqnarray}
    2x^2B+x^2 &=& 2aB^2+(a+2b)B+b \label{eq_lip_5}, \\ 
    3x^2B^2+2x^2B+x^2 &=& 3aB^3+(2a+3b)B^2+(a+2b)B+b \label{eq_lip_6}. 
  \end{eqnarray}
  From Equation~(\ref{eq_lip_5}) and $b_1=c$ we conclude $a+2b\equiv c\pmod B$. Since $0\le a+2b\le 3B-3<3B$ and $0\le c\le B-1$ we have 
  $a+2b\in\{c,c+B,c+2B\}$.
  \begin{itemize}
    \item If $a+2b=c$, then Equation~(\ref{eq_lip_6}) and $b_2=c$ gives $2a+3b\equiv c\pmod B$. From $a+2b=c$ we conclude $2a+3b=c+(a+b)\le c+(a+2b)\le c+c\le 
          c+(B-1)<c+B$, so that $2a+3b=c$ due to $2a+3b=c+(a+b)\ge c$. Subtracting the equations $a+2b=c$ and $2a+3b=c$ yields $a+b=0$, so that $a,b\ge 0$ gives 
          $a=b=0$. However, then $x^2=aB+b=0$ then implies $x=0$, which contradicts our assumption $x\ge 1$.
    \item If $a+2b=c+2B$, then Equation~(\ref{eq_lip_6}) and $b_2=c$ gives $2a+3b+2\equiv c\pmod B$. Since $2a+3b+2=(a+2b)+(a+b+2)=c+2B+(a+b+2)$ we conclude 
          $a+b+2\equiv 0\pmod B$. From $0<a+b+2\le 2B-1<2B$ we conclude $a+b+2=B$, so that $a+2b=c+2B$ yields $b+B-2=c+2B$, which is equivalent to $b-2=c+B$.  
          Since the left hand side is strictly smaller than $B$ and the right hand side is at least $B$, we obtain a contradiction.   
  \end{itemize}
  Thus, we have
  \begin{equation}
    a+2b=c+B \label{eq_lip_8}.
  \end{equation}
  Equation~(\ref{eq_lip_6}) and $b_2=c$ gives $2a+3b+1\equiv c\pmod B$. Since $2a+3b+1=(a+2b)+(a+b+1)=c+2B+(a+b+1)$ we conclude 
  $a+b+1\equiv 0\pmod B$. Since $0<a+b+1\le 2B-2<2B$, we have $a+b=B-1$, so that Equation~(\ref{eq_lip_8}) yields
  $$
    c=b-1\quad\text{and}\quad x^2\operatorname{div}B=a=B-1-b.  
  $$
  With this we have 
  $$
    x^2=aB+b=(B-1-b)B+b=(B-1)(B-b),
  $$
  so that $x^2$ is divisible by $B-1$. Finally, it remains to observe that if $b=0$, then $a=B-1$, $x^2=aB+b=B^2-B$ and 
  $$
    (B-1)^2=B^2-2B+1\overset{B\ge 2} < x^2 < B^2,
  $$
  which shows that no such integer $x$ exists.
\end{proof}

\begin{corollary}
  \label{cor_sequence_equal_digit}
  Let $B\in\mathbb{N}_{\ge 2}$, $n\in\mathbb{N}_{\ge 3}$, and $1\le x\le B-1$ be an integer such that the last but one and the last but second digit of 
  $$
    \big([\underset{n\text{ times}}{\underbrace{x\dots x}}]_B\big)^2
  $$
  coincide, then $B-1$ divides $x^2$ and the last $n$ digits of the square are given by 
  $$
    [\underset{n-1\text{ times}}{\underbrace{c\dots c}}\,b]_B,
  $$
  where $b=x^2\operatorname{rem} B$, $c=b-1$, $x^2\operatorname{div}B=B-1-b$, and $1\le b\le B-1$.
\end{corollary} 

From the Lemma~\ref{lemma_all_digit_equal_x_squared} and Lemma~\ref{lemma_all_digit_equal_x_squared_condition} we directly conclude:
\begin{proposition}
  \label{prop_sequence_equal_digit}
  Let $B\in\mathbb{N}_{\ge 2}$ and $1\le x\le B-1$ be an integer such that $B-1$ divides $x^2$. Then, we have
  $$
    \big([\underset{n\text{ times}}{\underbrace{x,\dots,x}}]_B\big)^2=[\underset{n-1\text{ times}}{\underbrace{B-b,\dots,B-b}}\,,B-1-b,\, 
    \underset{n-1\text{ times}}{\underbrace{b-1,\dots,b-1}}\,,b ]_B
  $$
  for each $n\in\mathbb{N}_{>0}$, where $b=x^2\operatorname{rem}B$ satisfying $1\le b\le B-1$.
\end{proposition} 


Note that for $B=10$ the term $B-1=9$ is a square itself, so that there are $\sqrt{B-1}$ integers $x$ such that $x^2$ is divisible by $B-1$ and $1\le x\le B-1$. Another example 
is given by $B=17$. In general $x=B-1$ is always a solution leading to
$$
  \big([\underset{n\text{ times}}{\underbrace{B-1,\dots,B-1}}]_B\big)^2=[\underset{n-1\text{ times}}{\underbrace{B-1,\dots,B-1}},\,B-2,\,\underset{n-1\text{ times}}{\underbrace{0,\dots,0}}\,,1]_B.
$$ 
If $B-1$ is \emph{squarefree}\footnote{An integer $n\in\mathbb{N}_{>0}$ is called squarefree if $k^2 | n$ implies $k=1$ for each positive integer $k$.}, then this is indeed the only 
possibility. Corresponding examples are given by $B=4$ and $B=11$. In general there are exactly $s$ solutions where $s$ is the largest integer such that $B-1=s^2t$ for some integer $t$.

By construction, the squares in Proposition~\ref{prop_sequence_equal_digit} typically consist of four different digits with respect to base $B$ if $n\ge 2$. However, by minor 
modifications we can achieve three different digits as e.g.\ shown in Equation~(\ref{ex_infinite_pattern}).  
In general we assume the form 
\begin{equation}
  \label{eq_sequence_equal_digit}
  N=s+B^c\cdot x\cdot \frac{B^n-1}{B-1}+r\cdot B^{c+n}, 
\end{equation} 
where we call $r$ prefix, $s$ suffix, and $c=l_B(s)$ is the number of digits of $s$ in base $B$. The possible choices for $x$ are completely characterized in 
Corollary~\ref{cor_sequence_equal_digit}. So, we may just loop over all positive integers $s,r$ up to a suitable upper bound, loop over all choices for $x$, 
set $n$ to the number of digits of $\max\{s,r\}$, and just check whether $N^2$ consists of just three different digits. In Table~\ref{table_sequence_equal_digit} 
we list the found suitable choices for the three parameters. 
\begin{table}[htp]
  \begin{center}
    \begin{tabular}{rrrr}
      \hline 
      $s$ & $r$ & $x$ & $D_{10}(N^2)$\\ 
      \hline
      4 & & 3 & $\{1,5,6\}$ \\
      5 &     & 3 & $\{1,2,5\}$ \\
      5 & 2 & 3 & $\{2,4,5\}$ \\
      5 & 123 & 3 & $\{1,2,5\}$ \\
      8 &     & 3 & $\{1,2,4\}$ \\
      59 & & 3 & $\{1,2,8\}$ \\
      504485 & & 3 & $\{1,2,5\}$ \\
      5 & & 6 & $\{2,4,5\}$ \\  
      5 & 1 & 6 & $\{2,5,7\}$\\
      7 & & 6 & $\{4,8,9\}$\\
      8 &  & 6 & $\{2,4,6\}$ \\
      515 & & 6 & $\{2,4,5\}$ \\ 
      7 & & 9 & $\{0,4,9\}$ \\ 
      \hline
    \end{tabular}
    \caption{Suitable parameters for Equation~(\ref{eq_sequence_equal_digit}).}
    \label{table_sequence_equal_digit}
  \end{center}
\end{table}      

As there are suitable suffixes as large as $s=504485$ we want to go a bit into details of a more sophisticated search algorithm. Starting from 
Equation~(\ref{eq_sequence_equal_digit}) we compute
$$
  N^2=s^2+2sxB^c\cdot\frac{B^n-1}{B-1}+2rs B^{c+n} + x^2B^{2c}\cdot\left(\frac{B^n-1}{B-1}\right)^2+2rxB^{2c+n}\cdot\frac{B^n-1}{B-1}+ r^2 B^{2c+2n}. 
$$
From here we already see that if $c$ is the number of digits of $s$, then the last $c$ digits of $s^2$ occur as digits of $N^2$. E.g.\ for $s=504485$ the 
three digits are forced to be contained in $\{1,2,5\}$ due to $[504485]^2=[254505|115225]$. In Section~\ref{sec_search_algorithms} we present an algorithm to compute a 
list of candidates for $s$.

Choosing $1\le x\le B-1$ such that $B-1$ divides $x^2$ and setting $b=x^2\operatorname{rem}B$ we can rewrite our equation for $N^2$ to
\begin{eqnarray}
  N^2&=&\left(s^2+bB^{2c}\right)+\left(2sx+(b-1)B^{c+1}\right)\cdot B^c\cdot\frac{B^{n-1}-1}{B-1}+\left(2sx+2rsB+(B-1-b)B^{c+1}\right)\cdot B^{c+n-1}\notag\\
  &&+\left(2rx+(B-b)B\right)\cdot B^{2c+n}\cdot\frac{B^{n-1}-1}{B-1}+\left(2rx+r^2B\right)\cdot B^{2c+2n-1}.
\end{eqnarray}
To further bring the formula in the direction of a base-$B$-representation we need a little auxiliary result.
\begin{lemma}
  Let $y\in\mathbb{N}_{>0}$ and $u,v$ be defined by $u+(B-1)v=y$, where $1\le u\le B-2$ and $v\in\mathbb{N}_{\ge 0}$. With this, let $l$ be the 
  smallest integer such that $u\cdot\frac{B^l-1}{B-1}\ge v$. Then, we have
  $$
    y\cdot\frac{B^n-1}{B-1}=\underset{\in[0,B^l-1]}{\underbrace{u\cdot\frac{B^l-1}{B-1}-v}}+\sum_{i=0}^{n-l-1} u\cdot B^{l+i} +v\cdot B^n
  $$
  for $n>l$, i.e., the attained digits are contained in $\varphi_B(u\cdot\frac{B^l-1}{B-1}-v)$, $\varphi_B(v)$, or equal to $u$. 
\end{lemma}
So having chosen $s$ and $x$ we can compute the occurring digits in the last $n+c-1$ places of $N^2$ without specifying $n$. Since 
$x<B$ we typically have $2rx\ll r^2$, so that the leading digits of $N^2$ are determined by the leading digits of $r^2$. In Section~\ref{sec_search_algorithms} we also 
present an algorithm to compute a list of candidates for $r$. 

We want to close this subsection by stating a specific parametric choice for $r$, $s$, and $x$.
\begin{lemma}
  \label{lemma_pattern_b_1_s}
  Let $B\ge 2$. If $x:=\sqrt{B-1}\in \mathbb{N}$, then $\big([B-1,\dots,B-1,B-x]_B\big)^2$ consists of digits in $\{0,B-2x,B-1\}$. 
  If $x:=\sqrt{B+1}-1\in \mathbb{N}$, then $\big([B-1,\dots,B-1,B-x]_B\big)^2$ consists of digits in $\{0,B-2x,B-1\}$.  
\end{lemma}
\begin{proof}
  We have $\left(B^k-x\right)^2=B^{2k}-2x\cdot B^k+x^2$. For $1\le x^2<B$ the digits in base $B$-representation are given by $0$, $B-1$, $B-2x$ and $x^2$, so that either 
  $B-1=x^2$ or $B-2x=x^2$.
\end{proof}

Is it always, i.e., for all bases $B$, possible to find a pattern of the form $[B-1,B-1,\dots,B-1,z]_B$ such that the square consists of only three digits? For $B=10$ we have to choose 
$z=7$. For $z=3$, i.e., the other root of $B-1=9$ modulo $10$, we end up with squares consisting of four different digits. For e.g.\ $B=11$ the number $B-1$ is not a quadratic 
residue.\footnote{An integer $m$ is called a quadratic residue modulo $n$, iff there exists an integer $x$ such that $x^2\equiv m\pmod n$.}

\subsubsection{Infinite patterns with blocks of zeros}
\label{subsec_infinite_zero-blocks}

If the repeating digit is a zero, then it seems to be much easier to construct some infinite patterns. An easy example is given by:
\begin{lemma}
  \label{lemma_zero_blocks_and_ones_1}
  For $B\ge 3$ and $n\ge 0$ we have
  $$
    \big([1,\underset{n\text{ times}}{\underbrace{0,\dots,0}},1]_{B}\big)^2=[1,\underset{n\text{ times}}{\underbrace{0,\dots,0}},2,\underset{n\text{ times}}{\underbrace{0,\dots,0}},1]_B.
  $$
\end{lemma}
\begin{proof}
  Let $N\in\mathbb{N}$ be the root on the left hand side, then we have $N=B^{n+1}+1$. Since the number $M$ on the right hand side is given by 
  $M=B^{2n+2}+2B^{n+1}+1$, we have $M=N^2$.
\end{proof}  

There are a lot of similar results with similar proofs (which we omit). Two further examples are given by:
\begin{lemma}
  \label{lemma_zero_blocks_and_ones_2}
  For $B\ge 3$ and $n\ge 1$ we have
  $$
    \big([1,\underset{n\text{ times}}{\underbrace{0,\dots,0}},1,1]_{B}\big)^2=[1,\underset{n\text{ times}}{\underbrace{0,\dots,0}},2,2,\underset{n-1\text{ times}}{\underbrace{0,\dots,0}},1,2,1]_B.
  $$
\end{lemma}

\begin{lemma}
  \label{lemma_zero_blocks_and_ones_3}
  For $B\ge 3$ and $n\ge 1$ we have
  $$
    \big([1,1,\underset{n\text{ times}}{\underbrace{0,\dots,0}},1]_{B}\big)^2=[1,2,1,\underset{n-1\text{ times}}{\underbrace{0,\dots,0}},2,2,\underset{n\text{ times}}{\underbrace{0,\dots,0}},1]_B.
  $$
\end{lemma}
In order to describe an underlying general construction pattern we use a prefix $r$ and a suffix $s$. 
\begin{lemma}
   \label{lemma_r_s_construction}
   Given two positive integers $r$ and $s$, we have
   \begin{equation}
    \big(\big[[\varphi_B(r),\underset{n\text{ times}}{\underbrace{0,\dots,0}},\varphi_B(s)\big]_B\big)^2=
    \big[\varphi_B(r^2),\underset{n+l_B(s)-l_B(2rs) \text{ times}}{\underbrace{0,\dots,0}},\varphi_B(2rs),\underset{n+l_B(s)-l_B(s^2)\text{ times}}{\underbrace{0,\dots,0}},\varphi_B(s^2)\big]
  \end{equation} 
  for all
  $n\ge \max\{l_B(2rs)-l_B(s),l_B(s^2)-l_B(s)\}=\max\{l_B(2rs),l_B(s^2)\}-l_B(s)\ge 0$.
\end{lemma}
\begin{proof}
  The root on the left hand side is given by $N=r\cdot B^{n+l_B(s)}+s$, so that 
  $$
    N^2=r^2\cdot B^{2n+2l_B(s)} +2rs\cdot B^{n+l_B(s)}+s^2.
  $$
\end{proof}
For the construction of infinite patterns with at most $3$ different digits (including zero), where all digits except $0$ occur only a finite number of times, we can proceed as follows. 
Given $r$ and $s$ the condition on the digits of the square is
\begin{equation}
  \left|\Big(D_B(r^2)\cup D_B(2rs)\cup D_B(s^2)\Big)\backslash \{0\}\right|\le 2.
\end{equation}
So, if we first compute a list of integers $s$ that are not {\lq\lq}too large{\rq\rq} satisfying $\left|D_B(s^2)\backslash \{0\}\right|\le 2$ than we can 
check the above condition for all pairs. Note that if there is such a pattern for $(r,s)$ then there is a corresponding {\lq\lq}reversed{\rq\rq} pattern for $(s,r)$, so that 
we assume $r\le s$ in the subsequent lists. Examples of this construction are given by
\begin{itemize}
  \item $r=[1]_B$, $s=[1]_B$, $r^2=[1]_B$, $2rs=[2]_B$, $s^2=[1]_B$, $D=\{0,1,2\}$ for $B\ge 3$;  
  \item $r=[1]_B$, $s=[1,1]_B$, $r^2=[1]_B$, $2rs=[2,2]_B$, $s^2=[1,2,1]_B$, $D=\{0,1,2\}$ for $B\ge 3$; 
  \item $r=[1]_B$, $s=[2]_B$, $r^2=[1]_B$, $2rs=[4]_B$, $s^2=[4]_B$, $D=\{0,1,4\}$ for $B\ge 5$;
  \item $r=[1]_B$, $s=[B-1]_B$, $r^2=[1]_B$, $2rs=[1,B-2]_B$, $s^2=[B-2,1]_B$, $D=\{0,1,B-2\}$ for $B\ge 3$;
  \item $r=[2]_B$, $s=[2]_B$, $r^2=[4]_B$, $2rs=[8]_B$, $s^2=[4]_B$, $D=\{0,4,8\}$ for $B\ge 9$;
  \item $r=[2]_B$, $s=[2,2]_B$, $r^2=[4]_B$, $2rs=[8,8]_B$, $s^2=[4,8,4]_B$, $D=\{0,4,8\}$ for $B\ge 9$.
\end{itemize}
We remark that e.g.\ $r=s=[1,1]_B$ is not in this list since $r^2=[1,2,1]_B$ and $2rs=[2,4,2]_B$. Note that in the above list the examples are valid 
for all $B$ that are sufficiently large and $\varphi_B(r^2),\varphi_B(2rs),\varphi_B(s^2)$ does not contain the zero digit. A different example is given by 
$r=[5]_{10}$, $s=[5]_{10}$ with $r^2=s^2=[2,5]_{10}$ and $2rs=[50]_{10}$, so that $D=\{0,2,5\}$. We can turn this example into a parametric example:
\begin{itemize}
  \item $r=[B/2]_B$, $s=[B/2]_B$, $r^2=[(B-2)/4,B/2]_B$, $2rs=[B/2,0]_B$, $s^2=[(B-2)/4,B/2]_B$, $D=\{0,B/2,(B-2)/4\}$ for all bases with $B\equiv 2\pmod 4$.
\end{itemize}
There are also examples where $[r]_B$ or $[s]_B$ contain at least one zero digit:
\begin{itemize}
  \item $r=[1,0,2]_{10}$, $s=[2,0,1]_{10}$, $r^2=[1,0,4,0,4]_{10}$, $2rs=[4,1,0,0,4]_{10}$, $s^2=[4,0,4,0,1]_{10}$, $D=\{0,1,4\}$ for $B=10$;
  \item $r=[5]_{10}$, $s=[5,0,5]_{10}$, $r^2=[2,5]_{10}$, $2rs=[5,0,5,0]_{10}$, $s^2=[2,5,5,0,2,5]_{10}$, $D=\{0,2,5\}$ for $B=10$.
\end{itemize}
While these two examples work for $B=10$ only so far, we can generalize them in a different way by inserting more zeros into $r$ and/or $s$:
\begin{itemize}
  \item $r=[1,0,0,2]_{10}$, $s=[2,0,0,1]_{10}$, $r^2=[1,0,0,4,0,0,4]_{10}$, $2rs=[4,0,1,0,0,0,4]_{10}$, $s^2=[4,0,0,4,0,0,1]_{10}$, $D=\{0,1,4\}$ for $B=10$;
  \item $r=[1,0,0,0,2]_{10}$, $s=[2,0,0,0,1]_{10}$, $r^2=[1,0,0,0,4,0,0,0,4]_{10}$, $2rs=[4,0,0,1,0,0,0,0,4]_{10}$, $s^2=[4,0,0,0,4,0,0,0,1]_{10}$, $D=\{0,1,4\}$ for $B=10$;
  \item $r=[5]_{10}$, $s=[5,0,0,5]_{10}$, $r^2=[2,5]_{10}$, $2rs=[5,0,0,5,0]_{10}$, $s^2=[2,5,0,5,0,0,2,5]_{10}$, $D=\{0,2,5\}$ for $B=10$;
  \item $r=[5]_{10}$, $s=[5,0,0,0,5]_{10}$, $r^2=[2,5]_{10}$, $2rs=[5,0,0,0,5,0]_{10}$, $s^2=[2,5,0,0,5,0,0,0,2,5]_{10}$, $D=\{0,2,5\}$ for $B=10$;
  \item $r=[5,0,5]_{10}$, $s=[5,0,0,5]_{10}$, $r^2=[2,5,5,0,2,5]_{10}$, $2rs=[5,0,5,5,0,5,0]_{10}$, $s^2=[2,5,0,5,0,0,2,5]_{10}$, $D=\{0,2,5\}$ for $B=10$.
\end{itemize}
We remark that these infinite patterns are not listed at \url{http://www.worldofnumbers.com/threedigits.htm}. Moreover, we can turn the first two examples 
into a $2$-parameter parametric series:
\begin{lemma}
  \label{lemma_1_2_construction_for_base_10}
  Given two integers with $k\ge 1$ and $n\ge k+1$
  we have
  \begin{eqnarray*}
  &&\big(\big[1,\underset{k\text{ times}}{\underbrace{0,\dots,0}},2,\underset{n\text{ times}}{\underbrace{0,\dots,0}},2,\underset{k\text{ times}}{\underbrace{0,\dots,0}},1\big]_{10}\big)^2\\ 
  &=&\big[1,\underset{k\text{ times}}{\underbrace{0,\dots,0}},4,\underset{k\text{ times}}{\underbrace{0,\dots,0}},4,\underset{n-k-1\text{ times}}{\underbrace{0,\dots,0}},4,\underset{k-1\text{ times}}{\underbrace{0,\dots,0}},1,\underset{k+1\text{ times}}{\underbrace{0,\dots,0}},
  4,\underset{n-k-1\text{ times}}{\underbrace{0,\dots,0}},4,\underset{k\text{ times}}{\underbrace{0,\dots,0}},4,\underset{k\text{ times}}{\underbrace{0,\dots,0}},1\big]_{10},
 \end{eqnarray*}
 and
 \begin{eqnarray*}
  &&\big(\big[2,\underset{k\text{ times}}{\underbrace{0,\dots,0}},1,\underset{n\text{ times}}{\underbrace{0,\dots,0}},1,\underset{k\text{ times}}{\underbrace{0,\dots,0}},2\big]_{10}\big)^2\\ 
  &=&\big[4,\underset{k\text{ times}}{\underbrace{0,\dots,0}},4,\underset{k\text{ times}}{\underbrace{0,\dots,0}},1,\underset{n-k-1\text{ times}}{\underbrace{0,\dots,0}},4,\underset{k-1\text{ times}}{\underbrace{0,\dots,0}},1,\underset{k+1\text{ times}}{\underbrace{0,\dots,0}},
  4,\underset{n-k-1\text{ times}}{\underbrace{0,\dots,0}},1,\underset{k\text{ times}}{\underbrace{0,\dots,0}},4,\underset{k\text{ times}}{\underbrace{0,\dots,0}},4\big]_{10},
 \end{eqnarray*}
 i.e., the squares consist of the digits $0$, $1$, and $4$ only.
\end{lemma}
\begin{proof}
  We set $B:=10$ and apply Lemma~\ref{lemma_r_s_construction} with
  \begin{eqnarray*}
    r &=& B^{k+1}+2=[1,\underset{k\text{ times}}{\underbrace{0,\dots,0}},2]_B\text{ and} \\ 
    s &=& 2\cdot B^{k+1}+1=[2,\underset{k\text{ times}}{\underbrace{0,\dots,0}},1]_B,
  \end{eqnarray*}  
  so that
  \begin{eqnarray*}
    r^2 &=& B^{2k+2}+4\cdot B^{k+1}+4=[1,\underset{k\text{ times}}{\underbrace{0,\dots,0}},4,\underset{k\text{ times}}{\underbrace{0,\dots,0}},4]_B,\\ 
    2rs &=& 4\cdot B^{2k+2}+10\cdot B^{k+1}+4=[4,\underset{k-1\text{ times}}{\underbrace{0,\dots,0}},1,0,\underset{k\text{ times}}{\underbrace{0,\dots,0}},4]_{10},\\ 
    s^2 &=& 4\cdot B^{2k+2}+4\cdot B^{k+1}+1=[4,\underset{k\text{ times}}{\underbrace{0,\dots,0}},4,\underset{k\text{ times}}{\underbrace{0,\dots,0}},1]_B,,\\ 
  \end{eqnarray*}
  $l_B(s)=k+2$, and $l_B(s^2)=l_B(2rs)=2k+3$.
  
  The second part can be proven analogously, i.e., it arises by interchanging the roles of $r$ and $s$ in Lemma~\ref{lemma_r_s_construction}.
\end{proof}
We can slightly generalize Lemma~\ref{lemma_1_2_construction_for_base_10}:
\begin{lemma}
  \label{lemma_1_2_construction_for_base_10_gen}
  Let
  $$
    N_1=\big[1,\underset{k\text{ times}}{\underbrace{0,\dots,0}},2,\underset{n\text{ times}}{\underbrace{0,\dots,0}},2,\underset{k\text{ times}}{\underbrace{0,\dots,0}},1\big]_{10} 
    =1+2\cdot 10^{k+1}+2\cdot 10^{n+k+2}+1\cdot 10^{n+2k+3}
  $$
  and
  $$
    N_2=\big[2,\underset{k\text{ times}}{\underbrace{0,\dots,0}},1,\underset{n\text{ times}}{\underbrace{0,\dots,0}},1,\underset{k\text{ times}}{\underbrace{0,\dots,0}},2\big]_{10} 
    =2+1\cdot 10^{k+1}+1\cdot 10^{n+k+2}+2\cdot 10^{n+2k+3}.
  $$
  Then, $N_1^2$ and $N_2^2$ have digits in $\{0,1,4\}$ in base $10$ if 
  $$
    k+n+2,\,\, 2k+2,\,\, 2k+n+4,\,\, 2k+2n+4,\text{ and }3k+n+4
  $$
  are pairwise different. The conditions are e.g.\ satisfied if $k>n$.
\end{lemma}
\begin{proof}
  We have
  $$
    N_1^2=1 + 4\cdot 10^{k+1} + 4\cdot 10^{k+n+2} + 4\cdot 10^{2k+2} + 1\cdot 10^{2k+n+ 4} + 4\cdot 10^{2k+2n+4} + 4\cdot 10^{3k+n+4}  + 4\cdot 10^{3k+2n+5} + 1\cdot 10^{4k+2n+6}
  $$
  and
  $$
    N_2^2=4 + 4\cdot 10^{k+1} + 4\cdot 10^{k+n+2} + 1\cdot 10^{2k+2} + 1\cdot 10^{2k+n+4}  + 1\cdot 10^{2k+2n+4} + 4\cdot 10^{3k+n+4}  + 4\cdot 10^{3k+2n+5} + 4\cdot 10^{4k+2n+6}.
  $$
  Note that the pairs of exponents that are not mentioned in the statement cannot coincide.
\end{proof}

The last three examples, before Lemma~\ref{lemma_1_2_construction_for_base_10}, are part of the following $3$-parameter series:
\begin{lemma}
  \label{lemma_half_of_basis_as_digits}
  Given a basis $B$ with $B\equiv 2\pmod 4$ and three integers with $l\ge 1$, $k\ge l+1$, and $n\ge k+2$
  we have
  \begin{eqnarray*}
  &&\big(\big[\tfrac{B}{2},\underset{k\text{ times}}{\underbrace{0,\dots,0}},\tfrac{B}{2},\underset{n\text{ times}}{\underbrace{0,\dots,0}},\tfrac{B}{2},\underset{l\text{ times}}{\underbrace{0,\dots,0}},\tfrac{B}{2}\big]_{B}\big)^2\\ 
  &=&\big[\tfrac{B-2}{4},\tfrac{B}{2},\underset{k-1\text{ times}}{\underbrace{0,\dots,0}},\tfrac{B}{2},\underset{k\text{ times}}{\underbrace{0,\dots,0}},\tfrac{B-2}{4},\tfrac{B}{2},\underset{n-k-2\text{ times}}{\underbrace{0,\dots,0}},\tfrac{B}{2},\underset{l\text{ times}}{\underbrace{0,\dots,0}},\\ 
  &&\tfrac{B}{2},\underset{k-l-1\text{ times}}{\underbrace{0,\dots,0}},\tfrac{B}{2},\underset{l\text{ times}}{\underbrace{0,\dots,0}},
  \tfrac{B}{2},\underset{n-l-1\text{ times}}{\underbrace{0,\dots,0}},\tfrac{B-2}{4},\tfrac{B}{2},\underset{l-1\text{ times}}{\underbrace{0,\dots,0}},\tfrac{B}{2},\underset{l\text{ times}}{\underbrace{0,\dots,0}},\tfrac{B-2}{4},\tfrac{B}{2}\big]_{B},
 \end{eqnarray*}
 i.e., the square consists of the digits $0$, $\tfrac{B-2}{4}$, and $\tfrac{B}{2}$ only.
\end{lemma} 
\begin{proof}
  We apply Lemma~\ref{lemma_r_s_construction} with
  \begin{eqnarray*}
    r &=& \tfrac{B}{2}\cdot B^{k+1}+\tfrac{B}{2}=[\tfrac{B}{2},\underset{k\text{ times}}{\underbrace{0,\dots,0}},\tfrac{B}{2}]_B\text{ and} \\ 
    s &=& \tfrac{B}{2}\cdot B^{l+1}+\tfrac{B}{2}=[\tfrac{B}{2},\underset{l\text{ times}}{\underbrace{0,\dots,0}},\tfrac{B}{2}]_B,
  \end{eqnarray*}  
  so that
  \begin{eqnarray*}
    r^2 &=& \tfrac{B^2}{4} \cdot B^{2k+2}+\tfrac{B^2}{2}\cdot B^{k+1}+\tfrac{B^2}{4}=[\tfrac{B-2}{4},\tfrac{B}{2},\underset{k-1\text{ times}}{\underbrace{0,\dots,0}},\tfrac{B}{2},\underset{k\text{ times}}{\underbrace{0,\dots,0}},\tfrac{B-2}{4},\tfrac{B}{2}]_{B},\\ 
    2rs &=& \tfrac{B^2}{2}\cdot B^{l+k+2}+\tfrac{B^2}{2}\cdot B^{k+1}+\tfrac{B^2}{2}\cdot B^{l+1}+\tfrac{B^2}{2}=[\tfrac{B}{2},\underset{l\text{ times}}{\underbrace{0,\dots,0}},\tfrac{B}{2},\underset{k-l-1\text{ times}}{\underbrace{0,\dots,0}},\tfrac{B}{2},\underset{l\text{ times}}{\underbrace{0,\dots,0}},\tfrac{B}{2},0]_{B},\\ 
    s^2 &=& \tfrac{B^2}{4} \cdot B^{2l+2}+\tfrac{B^2}{2}\cdot B^{l+1}+\tfrac{B^2}{4}=[\tfrac{B-2}{4},\tfrac{B}{2},\underset{l-1\text{ times}}{\underbrace{0,\dots,0}},\tfrac{B}{2},\underset{l\text{ times}}{\underbrace{0,\dots,0}},\tfrac{B-2}{4},\tfrac{B}{2}]_{B},\\ 
  \end{eqnarray*}
  $l_B(s)=l+2$, $l_B(s^2)=2l+4$, and $l_B(2rs)=l+k+4$.
\end{proof}

The sequence of Lemma~\ref{lemma_zero_blocks_and_ones_1}, Lemma~\ref{lemma_zero_blocks_and_ones_2}, and Lemma~\ref{lemma_zero_blocks_and_ones_3}, as well as the corollaries 
from  Lemma~\ref{lemma_r_s_construction}  choosing $r=s=[2]_B$ or $r=[2]_B$, $s=[2,2]_B$, are part of a rather 
general parametric construction. To this end we say that $A=\{a_0,\dots,a_n\}$ is a (finite) \emph{Sidon set}\footnote{Some authors also speak of \emph{Sidon sequences}. 
Note that finite Sidon sets are in one-to-one correspondence to  Golomb rulers.} iff the pairwise sums $a_i+a_j$ (for $i\le j$) are all different, i.e., 
$\left|\{a_i+a_j\,:\, 0\le i\le j\le n\}\right|={{n+2}\choose 2}$.
\begin{theorem}
  \label{thm_square_digits_and_double_square_digits_everywhere}
  Let $A=\left\{a_0,\dots,a_n\right\}$ be a non-empty Sidon set, $u\in \mathbb{N}_{\ge 1}$, and $B\in\mathbb{N}$ with $B\ge 2u^2+1$. Then, the square of 
  $$
    N=\sum_{i=0}^n u\cdot B^{a_i} 
  $$
  consists of digits in $\{0,u^2,2u^2\}$ in base $B$.
\end{theorem}
\begin{proof}
  We have
  $$
    N^2=\sum_{i=0}^n u^2\cdot B^{2a_i}\,+\,\sum_{0\le i<j\le n} 2u^2\cdot B^{a_i+a_j}.
  $$
  Since $A$ is a Sidon set, the exponents of $B$ are pairwise different, so that $1\le u^2<2u^2<B$ implies the stated result.
\end{proof}
The condition that $N$ is not divisible by $B$ is equivalent to $0\in A$. Note that \url{http://www.worldofnumbers.com/threedigits.htm} 
contains five parametric series of integers $N$ consisting of zeros and twos such that $N^2$ has base-$10$ digits contained in $\{0,4,8\}$. Four of these patterns consist 
of three $2$s and one of two $2$s, i.e.,
$$
  \big([2,\underset{n\text{ times}}{\underbrace{0,\dots,0}},2]_B\big)^2=[4,\underset{n\text{ times}}{\underbrace{0,\dots,0}},8,\underset{n\text{ times}}{\underbrace{0,\dots,0}},4]_B,
$$
where $n\ge 0$. While the latter pattern is exhaustive, the smallest missing example with three $2$'s is given by
$$
 \big([2,0,0,0,2,0,2]_{B}\big)^2=[4,0,0,0,8,0,8,0,4,0,8,0,4]_B
$$
for $B\ge 9$ and corresponds to the choice $u=2$ and $A=\{0,2,6\}$ in Theorem~\ref{thm_square_digits_and_double_square_digits_everywhere}.\footnote{At \url{www.asahi-net.or.jp/~KC2H-MSM/mathland/math02/math0203.htm} under no.~1 and no.~18  
patterns are mentioned that may hint to Theorem~\ref{thm_square_digits_and_double_square_digits_everywhere} for $u=1$ and $u=2$, respectively. In the first case the description is 
{\lq\lq}uppermost and lowermost digit is $1$, and properly arranged $0$, $1${\rq\rq}.} From Lemma~\ref{lemma_r_s_construction} 
it can be obtained choosing $r=[2]_B$ and $s=[2,0,2]_B$. However, patterns for $\{0,1,2\}$ are in one-to-one correspondence to patterns for $\{0,4,8\}$. 
\begin{lemma}
  \label{lemma_multiplied_triples}
  Let $0<a<b$ and $s>1$ be integers such that $bs^2<B$. Then, we have
  \begin{enumerate}
    \item[(1)] If $N$ is a positive integer such that $D_B(N^2)=\{0,a,b\}$, then $sN$ is a positive integer with $D_B((sN)^2)=\{0,as^2,bs^2\}$.  
    \item[(2)] If $N$ is a positive integer such that $D_B(N^2)=\{0,s^2a,s^2b\}$ and $s$ is squarefree, then $N/s$ is a positive integer with 
               $D_B((N/s)^2)=\{0,a,b\}$.
  \end{enumerate}  
\end{lemma}
\begin{proof}
  For part (1) we write $N^2$ uniquely as 
  \begin{equation}
    N^2 =\sum_{i=1}^l c_i\cdot B^{n_i}  
  \end{equation}  
  with $0\le n_1<n_2<\dots<n_l$ and $c_i\in\{a,b\}$ for all $1\le i\le l$. With this we have 
  $$
    (sN)^2 =\sum_{i=1}^l (s^2c_i)\cdot B^{n_i}.  
  $$ 
  For part (2) we observe that all digits as well as $B$ are divisible by $s$, so that $N^2$ is divisible by $s$. Since $s$ is squarefree, also $N$ is divisible by $s$. With this, we 
  can continue similar as for part (1).
\end{proof}

While it is rather easy to construct Sidon sets of arbitrary cardinality, take e.g.\ 
$A=\left\{2^0,2^1,\dots,2^n\right\}$, it is a hard and open problem to determine the minimum possible value of $\max\{a\in A\}$ if $A$ has to be a Sidon set of a given 
cardinality. For {\lq\lq}length{\rq\rq} $n=\max\{a\in A\}=6$ the smallest Sidon set has cardinality four and is given by $A=\{0,1,4,6\}$. In our context this leads to
$$
 \big([2,0,2,0,0,2,2]_{B}\big)^2=[4,0,8,0,4,8,8,8,8,0,4,8,4]_B
$$
for $B\ge 9$. The sequence of the smallest possible lengths per cardinality of the Sidon set is given in \url{https://oeis.org/A003022} (using the relation to Golomb rulers). 
See also \url{https://blogs.distributed.net/2014/02/} for the current state of knowledge.

In order to generalize Lemma~\ref{lemma_half_of_basis_as_digits} we need the following generalization of a Sidon set:\footnote{Other generalizations in the literature consider the 
situation where the number of solutions of $a_i+a_j=m$ with $i\le j$ is upper bounded by some constant that may be larger than $1$ and/or sums with more summands are considered.}
\begin{definition}
  For a positive integer $t$ a set $A=\left\{a_0,\dots,a_n\right\}$ is called a $t$-distance Sidon set if 
  $$
    \Big|\left(a_i+a_j\right)-\left(a_u+a_v\right)\Big| \ge t 
  $$
  for all $0\le i\le j\le n$ and $0\le u\le v\le n$ with $(i,j)\neq (u,v)$.
\end{definition}
An example of a $2$-distance Sidon set with cardinality four is given by $A=\{0,2,8,12\}$. $1$-distance Sidon sets are just Sidon sets and from a $t$-distance Sidon set $A$ 
we can construct an $rt$-distance set $rA:=\{r\cdot a\,:\,a\in A\}$ for every positive integer $r$.
\begin{theorem}
  \label{thm_half_of_basis_as_digits}
  Let $A=\left\{a_0,\dots,a_n\right\}$ be a non-empty $2$-distance Sidon set and $B\in \mathbb{N}_{\ge 1}$ with $B\equiv 0\pmod 2$. Then, the square of 
  $$
    N=\sum_{i=0}^n \frac{B}{2}\cdot B^{a_i} 
  $$
  consists of digits in $\left\{0,\tfrac{B-2}{4},\tfrac{B}{2}\right\}$ in base $B$ if $B\equiv 2\pmod 4$.\footnote{At \url{www.asahi-net.or.jp/~KC2H-MSM/mathland/math02/math0203.htm} under no.~15 the case $B=10$ is adressed. The description is 
{\lq\lq}uppermost and lowermost digit is $5$, and properly arranged $0$, $5$.{\rq\rq}.} 
\end{theorem}
\begin{proof}
  We have
  $$
    N^2=\sum_{i=0}^n \frac{B^2}{4}\cdot B^{2a_i}\,+\,\sum_{0\le i<j\le n} \frac{B^2}{2}\cdot B^{a_i+a_j}.
  $$
  Since $A$ is a $2$-distance Sidon set, the exponents of $B$ are at pairwise distance at least two, the to-digit numbers $\tfrac{B^2}{4}$ and $\tfrac{B^2}{2}$ do not interfere. 
  For $B\equiv 2\pmod 4$ we have 
  $$
    \frac{B^2}{4}=\frac{B-2}{4}\cdot B+\frac{B}{2}.
  $$
  Additionally we have 
  $$
    \frac{B^2}{2}=\frac{B}{2}\cdot B+0.
  $$  
\end{proof}  
Our $2$-distance Sidon set $A=\{0,2,8,12\}$ gives the example
$$
  \big([5000500000505]_{10}\big)^2=[25005000255050505000255025]_{10}.
$$

Another way to apply Lemma~\ref{lemma_r_s_construction} is to choose as parameters $r$ and $s$ expressions that are obtained by applying Lemma~\ref{lemma_r_s_construction} 
to $\left(r_1,s_1\right)$ and $\left(r_2,s_2\right)$. As an example we will consider $\left(r_1,s_1\right)=\left(r_2,s_2\right)=(1,1)$. In both cases we can choose a 
suitable number $n$ for the length of the block of zeros, where \textit{the} $n$ may be different. So, let us choose the parameterization $r=1\cdot B^{k+1}+1$ and 
$s=1\cdot B^{l+1}+1$. By construction it is clear that both $r^2$ and $s^2$ consist of digits in $\{0,1,2\}$ only. So, it remains to check $2rs$, where we compute
$$
  2rs=2\cdot\left(1\cdot B^{k+1}+1\right)\cdot\left(1\cdot B^{l+1}+1\right)=2\cdot B^{k+l+2}+2B^{k+1}+2B^{l+1}+2.
$$   
So, if $l\neq k$ and $B\ge 3$, then $2rs$ consists of digits in $\{0,2\}$ only. For $l=k$ and $B\in\{3,4\}$ the number $2rs$ consists of digits in $\{0,1,2\}$ only. 
The first case is covered by Theorem~\ref{thm_square_digits_and_double_square_digits_everywhere}, while the second is a bit special since it is valid for $B\in\{3,4\}$ only. 
For $u>1$ in Theorem~\ref{thm_square_digits_and_double_square_digits_everywhere} we can also take $l=k$ if $B\in\{2u^2-1,2u^2\}$.\\ 
Choosing $\left((r_1,s_1\right)=(1,2)$ and $\left((r_1,s_1\right)=(2,1)$, we end up with Lemma~\ref{lemma_1_2_construction_for_base_10} using
$$
  2\cdot\left(1\cdot B^{k+1}+2\right)\cdot\left(2\cdot B^{l+1}+1\right)=4\cdot B^{k+l+2}+ 2\cdot B^{k+1}+8\cdot B^{l+1}+4.
$$   
Since the digits $\{0,1,4\}$ occur anyway, we have to choose $l=k$ and $B\in\{2,3,5,6,9,10\}$, so that $10\operatorname{rem} B\in\{0,1,4\}$, or $B=4$ with digits in $\{0,1,2\}$.\\ 
While choosing $r=1$ and $s=2$ works in any basis $B\ge 5$, we can choose $r=1$ and $s=v$ for all $v\ge 3$ and $B\in\{v^2-2v,v^2-1\}$. Of course, the roles of $r$ and $s$ 
can be interchanged.\\ 
Choosing $\left(r_1,s_1\right)=(1,B-1)$ and $\left(r_2,s_2\right)=(B-1,1)$ leads to 
\begin{eqnarray*}
  &&2\cdot\left(1\cdot B^{k+1}+(B-1)\right)\cdot\left((B-1)\cdot B^{l+1}+1\right)\\ &=&B^{k+l+3}+(B-2)\cdot B^{k+l+2}+
  \left(2B^{k+1}+2B^{l+3}-4B^{l+2}+2B^{l+1}\right)
  +1\cdot B+(B-2),
\end{eqnarray*}
so that the question remains for which $k$, $l$ and $B$ the term 
$$
  2B^{k+1}+2B^{l+3}-4B^{l+2}+2B^{l+1}
$$
has only digits in $\{0,1,B-2\}$ in base-$B$ representation. If $B\le 3$ then we cannot have more than three different digits, so that we assume 
$B\ge 4$ in the following. If $B=4$ and $k=l+2$, then the digit $3$ occurs while $B-2=2\neq 3$. All other cases with $B=4$ and $k\neq l+2$ 
the occurring digits are $\{0,1,2\}$. So, we assume $B\ge 5$ in the following. If $k>l+2$ or $k<l$, then the occurring non-zero digits are given by 
$\{0,1,2,B-2,B-4\}$ which has cardinality at least four for $B\ge 5$. If $k=l$, then the occurring non-zero digits are given by 
$\{0,1,4,B-2,B-4\}$ which has cardinality at least four for $B\ge 5$. If $k=l+1$, then the occurring non-zero digits are given by 
$\{0,1,2,3,B-2,B-4\}$ which has cardinality at least four for $B\ge 4$. If $k=l+1$, then the occurring non-zero digits are given by 
$\{0,1,2,B-2,B-4\}$ which has cardinality at least four for $B\ge 5$.
 
Choosing $\left(r_1,s_1\right)=\left(r_2,s_2\right)=\left(\tfrac{B}{2},\tfrac{B}{2}\right)$ is covered by Theorem~\ref{thm_half_of_basis_as_digits}.

Of course, Lemma~\ref{lemma_r_s_construction} can be generalized to two blocks of zero of length $n$ each
$$
  \big(\big[[\varphi_B(r),\underset{n\text{ times}}{\underbrace{0,\dots,0}},\varphi_B(s),\underset{n\text{ times}}{\underbrace{0,\dots,0}},\varphi_B(t)\big]_B\big)^2
$$
or $l$ blocks of zero of length $n$ each
$$
  \big(\big[[\varphi_B(r_0),\underset{n\text{ times}}{\underbrace{0,\dots,0}},\varphi_B(r_1),\dots,\underset{n\text{ times}}{\underbrace{0,\dots,0}},\varphi_B(r_l)\big]_B\big)^2.
$$
We can try something slightly more general for non-negative integers $a_0,\dots,a_l$ and positive integers $r_0,\dots,r_l$ we set 
$$
  N=\sum_{i=0}^l r_i\cdot B^{a_i}, 
$$
so that 
$$
  N^2=\sum_{i=0}^l r_i^2\cdot B^{2a_i}\,+\, \sum_{0\le i<j\le l} 2r_ir_j\cdot B^{a_i+a_j}.
$$  
Setting
$$
  t=\max\!\left\{l_B(r_i^2),l_B(2r_ir_j)\,:\, 0\le i\le l, i<j\le l\right\}
$$
we can choose $A:=\left\{a_0,\dots,a_l\right\}$ such that $A$ is a $t$-distance Sidon set. If $D_B(r_i^2)\subseteq D$ for all $0\le i\le l$, 
$D_B(2r_ir_j)\subseteq D$ for all $0\le i<j\le n$, and $0\in D$, then we have $D_B(N^2)\subseteq D$ for some subset $D\subseteq \{0,1,\dots,B-1\}$. Obviously, the conditions are more likely 
to be satisfied if several of the $r_i$ coincide. While this is a nice general construction strategy, we have to remark that the only known cases with $|A|\ge 3$ where it can be 
applied satisfy $r_0=r_1=\dots=r_l$ and indeed are covered by Theorem~\ref{thm_square_digits_and_double_square_digits_everywhere} and Theorem~\ref{thm_half_of_basis_as_digits}. Having 
such a general patter for all $t$-distance Sidon sets $A$ of a given cardinality is maybe too ambitious if $|A|\ge 3$. Another approach is to choose 
$a_i=i\cdot n+c_i$ for a parameter $n$ and constants $c_0,\dots,c_l\in \mathbb{Z}$. This way the ${{l+2}\choose 2}$ terms $r_i^2\cdot B^{2a_i}$ and $2r_ir_j\cdot B^{a_i+a_j}$ clearly 
interfere in a base-$B$ representation. However, the $r_0,\dots,r_l$ and the $c_0,\dots,c_l$ might be chosen in such a way that we will find patterns with just two non-zero digits. It 
is very likely that such pattern will be valid for specific bases only. We want to start to study the situation for $l=2$ in more detail.
\begin{lemma}
  \label{lemma_two_blocks_of_zeros}
  Let $B\in\mathbb{N}_{\ge 2}$, $r_0,r_1,r_2\in \mathbb{N}_{>0}$, $c\in\mathbb{Z}$, $D\subseteq\{0,1,\dots,B-1\}$ with $0\in D$, and
  $$
    N:=r_2\cdot B^{2n+c}+r_1\cdot B^n+r_0  
  $$
  for some integer $n$. If $D_B(r_2^2)\subseteq D$, $D_B(2r_1r_2)\subseteq D$, $D_B(2r_1r_0)\subseteq D$, $D_B(r_0^2)\subseteq D$, and $D_B(2r_0r_2\cdot B^c+r_1^2)\subseteq D$ 
  (if $c\ge 0$) or $D_B(2r_0r_2+r_1^2\cdot B^{-c})\subseteq D$ (if $c\le 0$), then we have $D_B(N^2)\subseteq D$ for all
  \begin{equation}
    n\ge \max\!\left\{l_B(r_0^2),l_B(2r_1r_0)-c,l_B(2r_0r_2+r_1^2B^{-c})-c,l_B(2r_1r_2)-c\right\}
  \end{equation}
  if $c\le 0$ and
  \begin{equation}
    n\ge \max\!\left\{l_B(r_0^2),l_B(2r_1r_0),l_B(2r_0r_2B^c+r_1^2)-c,l_B(2r_1r_2)-c\right\}
  \end{equation}
  if $c\ge 0$.
\end{lemma}
\begin{proof}
  We compute 
  $$
    N^2=r_2^2\cdot B^{4n+2c}+2r_1r_2\cdot B^{3n+c}+\left(2r_0r_2B^c+r_1^2\right)\cdot B^{2n}+2r_1r_0\cdot B^n+r_0^2.
  $$
  If $n$ is at least as large as requested, then the terms before the powers of $B$ do not interfere in base-$B$ representation.
\end{proof}

Below we give some examples for the application of Lemma~\ref{lemma_two_blocks_of_zeros}:
\begin{itemize}
  \item[(1)] $r_0=r_2=1$, $r_1=3$, $c=0$, $D=\{0,1,6\}$ for $n\ge 2$ and $B\in\{5,6,10,11\}$ (for $B=5$ we need $n\ge 3$ and have $D=\{0,1,2\}$; for $B=6$ we have $B\in\{0,1,5\}$);\\[-9mm]
  \item[(2)] $r_0=r_2=1$, $r_1=8$, $c=0$, $D=\{0,1,6\}$ for $n\ge 2$ and $B\in\{4,8,10,50,65,66\}$ (for $B\in\{50,65,66\}$ we have $D=\{0,1,16\}$; or $B\in\{4,8\}$ we 
             have $D=\{0,1,2\}$, where for $B=4$ we need $n\ge 3$);\\[-9mm]
  \item[(3)] $r_0=r_2=1$, $r_1=4$, $c=0$, $D=\{0,1,8\}$ for $n\ge 2$ and $B\in\{4,5,8,10,17,18\}$ (for $B\in\{4,8\}$ we have $D=\{0,1,2\}$ and for $B=5$ we have $D=\{0,1,3\}$);\\[-9mm]           
  \item[(4)] $r_0=r_2=1$, $r_1=9$, $c=1$, $D=\{0,1,8\}$ for $n\ge 2$ and $B\in\{4,5,9,10\}$ (for $B\in\{4,9\}$ we have $D=\{0,1,2\}$ and for $B=5$ we have $D=\{0,1,3\}$);
  \item[(5)] $r_0=r_2=2$, $r_1=6$, $c=0$, $D=\{0,2,4\}$ for $n\ge 2$ and $B\in\{10,11\}$;
  \item[(6)] $r_0=r_2=2$, $r_1=1$, $c=0$, $D=\{0,4,9\}$ for $n\ge 1$ and $B\ge 8$ or $B\in\{4,5\}$ (for $B\in\{5,8,9\}$ we have $D=\{0,1,4\}$ and for $B=4$ we have $D=\{0,1,2\}$);
  \item[(7)] $r_0=r_2=4$, $r_1=127$, $c=0$, $D=\{0,1,6\}$ for $n\ge 5$ and $B=10$;
  \item[(8)] $r_0=r_2=8$, $r_1=254$, $c=0$, $D=\{0,4,6\}$ for $n\ge 5$ and $B=10$;
  \item[(9)] $r_0=r_2=15$, $r_1=85$, $c=2$, $D=\{0,2,5\}$ for $n\ge 4$ and $B=10$.
\end{itemize}

Looking at the cases (1)-(4) we might look at the parametric case $r_0=r_2=1$, $r_1=x$, and $c=0$. Here we have $\{0,1\}\subseteq D$ and the remaining digit is determined by 
the base $B$-representation of $2x$ and $2+x^2$. If $B>2x$, then we have the possibilities $B\in\{x^2-2x+2,x^2+1,x^2+2\}$, which gives $x\in\{3,4\}$ for $B=10$. For 
$x<B\le 2B$ we can perform similar computations and obtain the solution $x=B+1-\sqrt{B-1}$ if $B-1$ is a square. If $|c|$ is sufficiently large, then the digits 
of $2r_0r_2B^c+r_1^2$ are given by the digits of $2r_0r_2$ and $r_1^2$, i.e. $D=\{0,1,2\}$ and $r_1^2$ has to consist of digits in $\{0,1,2\}$. Clearly we can choose $r_1=1$, but 
then the situation is covered by Theorem~\ref{thm_square_digits_and_double_square_digits_everywhere}.

The parameters of e.g.\ case (9) look a bit \textit{sporadic} so that the question arises who we can algorithmically find those cases for a given base $B\ge 4$:
\begin{itemize}
  \item[Step 1] Compute a list/set $L$ of positive integers $M$ such that $\left|D_B(M^2)\backslash\{0\}\right|\le 2$. 
  \item[Step 2] Loop over all $r_0,r_2\in L$ with $\left|\left(D_B(r_0^2)\cup D_B(r_1^2)\right)\backslash\{0\}\right|\le 2$. Set $D:=D_B(r_0^2)\cup D_B(r_1^2)\cup\{0\}$.
  \item[Step 3] Compute a list/set $L'$ of positive integers $r_1$ such that $r_1^2$ consists of digits in $D$ in base $B$-representation except for 
                a sequence of consecutive digits of length at most $l_B(2r_0r_2)+1$.
  \item[Step 4] Loop over all $c\in\mathbb{Z}$ with $|c|\le l_B(2r_0r_2)+1$ and all $r_1\in L'$. Check whether
                $$
                  D_B\!\left(
                  \left(2r_0r_2B^c+r_1^2\right)\cdot B^{\max\{0,-c\}}
                  \right)\subseteq D
                $$              
                and eventually output the parameters.
\end{itemize}    
We remark that so far all found examples satisfy $r_0=r_2$. While the presented technique is in principle applicable for $l\ge 3$ blocks of zeros of roughly length $n$, 
we have not found such an example that is not covered by the construction using $t$-distance Sidon sets, i.e., where the blocks all have lengths in terms of pairwise 
different parameters (modulo some constraints). 

\medskip

For base $B=10$ and $D=\{0,1,2\}$ another, so far unknown,\footnote{At \url{www.asahi-net.or.jp/~KC2H-MSM/mathland/math02/math0203.htm} under no.~2 some examples of the discussed 
type are listed. The description is {\lq\lq}numbers of $0$s between $1$ and $1$, and $1$ and $4$ is an odd number{\rq\rq}. However, we e.g.\ have $1010000010004000100000101^2 =
1020100020208080302080222020802030808020200010201$.} infinite pattern can be constructed. To this end let us consider the ansatz
\begin{equation}
  N=4\cdot 10^x+\sum_{i=0}^l 10^{n_i}
\end{equation} 
for some integers $x$ and $0=n_0<n_1<\dots<n_l$, so that
\begin{equation}
  N^2 = \sum_{i=0}^l 10^{2n_i} \,+\, \sum_{0\le i<j\le l} 2\cdot 10^{n_i+n_j} \,+\, \sum_{i=0}^l 8\cdot 10^{x+n_i}\,+\, 16\cdot 10^{2x}.  
\end{equation}   
In it easy to see that digits $0$, $1$, and $2$ will occur, so that we need to get rid of the terms with multipliers $8$ and $16$. Since 
$8+2=10$ and $16+2\cdot 2=20$, the idea is to choose $x$ and the $n_i$ in such a way that there exist indices $a_{i,1}$, $a_{i,2}$ with 
$x+n_i=n_{a_{i,1}}+n_{a_{i,2}}$ for $0\le i\le l$ and two different additional solutions of $2x=n_{i'}+n_{j'}$ for suitable indices 
$0\le i',j'\le l$. For $l=5$ an example is given by $x + n_0 = n_1 + n_2$, $x+n_1= n_0 + n_3$, $x+n_2= n_0 + n_4$, $x + n_3 = n_1 + n_5$, 
$x+n_4 = n_2 + n_5$, $x+n_5 =n_3+n_4$, and $2x = n_0 + n_5=n_1 + n_4$, where $n_0=0$. Solving the corresponding equation system gives a $2$-parameter 
solution. It remains to ensure that the exponents of the remaining terms are pairwise different.
\begin{lemma}
  \label{lemma_many_ones_1_four}
  Let $1\le n_1<n_2$ be integers, $n_0=0$, $n_3=2n_1+n_2$, $n_4=n_1+2n_2$, $n_5=2n_1+2n_2$, $x=n_1 + n_2$, and
\begin{equation}
  N=4\cdot 10^x+\sum_{i=0}^5 10^{n_i}.
\end{equation} 
If $0$, $n_1$, $2n_1$, $1+n_1+n_2$, $1+2n_1+n_2$, $3n_1+n_2$, $2n_2$, $1+n_1+2n_2$, $2n_1+2n_2$, $1+2n_1+2n_2$, $1+3n_1+2n_2$, $4n_1+2n_2$,  
$n_1+3n_2$, $1+2n_1+3n_2$, $1+3n_1+3n_2$, $4n_1+3n_2$, $2n_1+4n_2$, $3n_1+4n_2$, and $4n_1+4n_2$ are pairwise different integers, the $N^2$ contains 
only the digits $0$, $1$, and $2$ in base $10$. This condition is e.g.\ satisfied for $n_1\ge 2$, $n_2\ge 3n_1+1$.   
\end{lemma}  
\begin{proof}
  We compute 
  \begin{eqnarray*}
    N^2 &=& 1 + 2\cdot 10^{3n_1+n_2} + 10^{4n_1 + 4n_2} + 22\cdot 10^{2n_1+2n_2}+10^{1+n_1+2n_2}+10^{1+2n_1+n_2}+2\cdot 10^{3n_1+4n_2} \\ 
         && + 2\cdot 10^{4n_1 + 3n_2} + 10^{1+3n_1+3n_2} + 10^{1+2n_1+3n_2} + 10^{1+3n_1+2n_2}+10^{2n_1+4n_2} + 2\cdot 10^{n_1+3n_2} \\  
         &&   + 10^{4n_1 + 2n_2} + 10^{2n_2} + 10^{2n_1} + 2\cdot 10^{n_1} + 2\cdot 10^{n_2} + 10^{1+n_1+n_2}.
  \end{eqnarray*} 
\end{proof}
The number of occurring ones in the base $10$-representation of $N^2$ is $12$ and the number of occurring twos is $8$. An example is given by $\left(n_1,n_2\right)=(2,7)$, i.e.,
$$
  \left([1010000104010000101]_{10}\right)^2=[1020100210100211022100121010020010201]_{10}.
$$
We remark that also other pairs $\left(n_1,n_2\right)$, violating $n_2\ge 3n_1+1$, satisfy the condition on the pairwise different exponents in Lemma~\ref{lemma_many_ones_1_four}. 
Examples are e.g.\ given by $(3,8)$, $(4,6)$, $(4,10)$, $(4,11)$, $(5,7)$, $(5,8)$, $(5,12)$, $(5,13)$, and $(5,14)$.

Actually, Lemma~\ref{lemma_many_ones_1_four} is just an instance of the more general idea outlined before and the question arises if more such results can be found. We remark 
that this is impossible if $l\le 3$. However, we conjecture that for all $l\ge 5$ such patterns can be found. 

\subsubsection{{\lq\lq}Intricate{\rq\rq} patterns}
\label{subsec_intricate_patterns}
A very interesting type of patterns found by the Iranian mathematician Farideh Firoozbakht (1962--2019) are so called {\lq\lq}intricate{\rq\rq} patterns, c.f.\  
\begin{center}\url{http://www.worldofnumbers.com/threedigits.htm}.\end{center} As a shorthand notation we write $\langle a_1',\dots,a_l'\rangle$ for
$$
  4 \underset{a_1'\text{times}}{\underbrace{9\dots9}} 4 \underset{a_2'\text{times}}{\underbrace{9\dots9}}\dots 4 \underset{a_l'\text{times}}{\underbrace{9\dots9}}5, 
$$
where the $a_i'$ are non-negative integers. Only sum choice lead numbers $N$ with $D_{10}(N^2)\subseteq\{0,2,5\}$. Examples are e.g.\ given by $\langle 0\rangle$, $\langle 0,1\rangle$, 
$\langle 0,1,3\rangle$, $\langle 0,1,3,7\rangle$, $\langle 0,1,3,7,15\rangle$, \dots. A few more irregular such examples are known and described parametrically. 
However, examples as e.g.\ $\langle 0,1,3,2,10\rangle$ and $\langle 0,1,3,2,4,9\rangle$ were unknown so far. The question arises for which (finite) sequences $a_i'$ we end up 
with a square consisting of digits in $\{0,2,5\}$. To this end we introduce  a different representation. Let $a_1,\dots,a_l$ be pairwise different positive integers and
\begin{equation}
  \label{eq_intricate_representation}
  N:=\sum_{i=1}^l \left(5\cdot 10^{a_i}-10^{a_{i+1}+1}\right)\,+\, 5,
\end{equation}
where $a_{l+1}:=0$. For $\left(a_1,a_2,a_3\right)=(7,6,4)$ we have $N=44949995$ and $N^2=2020502050500025$. 
We can easily convert a sequence $a_1,a_2,\dots,a_l$ with $a_1>a_2>\dots>a_l>0$ to a representation $\langle a_1',\dots,a_l'\rangle$ and vice versa:
\begin{equation}
  a_i=\sum_{j=i}^la_j'\,+\, \left(l-i+1\right)
\end{equation}   
and
\begin{equation}
  a_i'=a_i-a_{i+1}-1
\end{equation}
for all $1\le i\le l$, where we set $a_{l+1}=0$. Next we evaluate and simplify the expression for $N^2$, when $N$ is given as specified in 
Equation~(\ref{eq_intricate_representation}), in terms of the $a_i$ for small values of $l$:
\begin{eqnarray*}
  N^2 &=& 25\cdot 10^{2a_1}  + 25 - 5\cdot 10^{a_1+1}\\ 
  N^2 &=& 25\cdot 10^{2a_1}+25\cdot 10^{2a_2}+25  - 5\cdot 10^{a_1 + a_2+1}  - 5\cdot 10^{a_1+1}  + 5\cdot 10^{a_2+1}\\ 
  N^2 &=& 25\cdot 10^{2a_1}+25\cdot 10^{2a_2}+25\cdot 10^{2a_3}+25-5\cdot 10^{a_1+a_2+1} -5\cdot 10^{a_1+a_3+1}-5\cdot 10^{a_1+1} \\ 
  && +5\cdot 10^{a_2+a_3+1} +5\cdot 10^{a_2+1}+5\cdot 10^{a_3+1} \\ 
  N^2 &=& 25\cdot 10^{2a_1}+25\cdot 10^{2a_2}+25\cdot 10^{2a_3}+25\cdot 10^{2a_3}+25\\ 
  && -5\cdot 10^{a_1+a_2+1} -5\cdot 10^{a_1+a_3+1}-5\cdot 10^{a_1+a_4+1}-5\cdot 10^{a_1+1} \\ 
  && +5\cdot 10^{a_2+a_3+1} +5\cdot 10^{a_2+a_4+1}+5\cdot 10^{a_3+a_4+1}+5\cdot 10^{a_2+1}+5\cdot 10^{a_3+1}+5\cdot 10^{a_4+1} 
\end{eqnarray*}  
Since we only want to use the digits in $\{0,2,5\}$ the terms with a {\lq\lq}$-5${\rq\rq}-coefficient have to cancel out by a suitable choice of the $a_i$. W.l.o.g.\ we 
assume $a_1> a_2> \dots> a_l$. For $l=1$ we need $2a_1=a_1+1$, so that $a_1=1$, $N=45$, and $N^2=2025$. For $l=2$ the ordering of the $a_1>a_2$ forces the cancellation 
given by $2a_1=a_1+a_2+1$, i.e, $a_2=a_1-1$. Since $a_1+1>a_2+1$ and $a_1+1>0$, we also need $2a_2=a_1+1$, so that $a_1=3$, $a_2=2$, $N=4495$, and $N^2=20205025$. For $l=3$, 
choosing $2a_1 = a_1 + a_2 + 1$, $2a_2 = a_1 + a_3 + 1$, and $2a_3 = a_1 + 1$ yields the unique solution $a_1=7$, $a_2=6$, $a_3=4$ leading to $N=44949995$ and 
$N^2=2020502050500025$. In general we have
\begin{equation}
  N^2=\sum_{i=1}^l 25\cdot 10^{2a_i}\,+\,25-\sum_{i=2}^l 5\cdot 10^{a_1+a_i+1}\,- 5\cdot 10^{a_1+1}\,+\,\sum_{2\le i<j\le l} 5\cdot 10^{a_i+a_j+1}\,+\, 
  \sum_{i=2}^l 5\cdot 10^{a_i+1}.
\end{equation}
For larger $l$ there are several possibilities to match the {\lq\lq}-5{\rq\rq}-exponents with the {\lq\lq}+25{\rq\rq}- or {\lq\lq}+5{\rq\rq}-exponents. The choice 
$2a_i=a_1+a_{i+1}+1$ for $1\le i\le -1$ and $2a_l=a_1+1$ leads to $a_i=2^l-2^{i-1}$ for all $1\le i\le l$. It can be easily check that the remaining $2$s and $5$s do not 
interfere.   

In the remaining part we return to the $\langle a_1',\dots,a_l'\rangle$-notation again. As mentioned, we have just one possibility for each $l$ with $1\le l\le 3$, 
i.e., $\langle 0\rangle$, $\langle 0,1\rangle$, and $\langle 0,1,3\rangle$.  For $l=4$ we have the two solutions $\langle 0,1,3,7\rangle$ and $\langle 0,1,3,2\rangle$ 
(or equivalently $15,14,12,8$ and $10,9,7,3$). Let us denote the number of solutions by $\beta_l$. From computer searches we know $\beta_1=1$, $\beta_2=1$, $\beta_3=1$, 
$\beta_4\ge 2$, $\beta_5\ge 5$, $\beta_7\ge 40$, and $\beta_8\ge 125$.
\begin{itemize}
  \item $l=5$: {\footnotesize $\langle 0,1,3,2,4\rangle$, $\langle 0,1,3,2,10\rangle$, $\langle 0,1,3,7,2\rangle$, $\langle 0,1,3,7,6\rangle$, $\langle 0,1,3,7,15\rangle$;}\\[-8mm]
  \item $l=6$: {\footnotesize $\langle 0,1,3,2,4,9\rangle$, $\langle 0,1,3,2,4,15\rangle$, $\langle 0,1,3,2,10,9\rangle$, $\langle 0,1,3,2,10,21\rangle$, $\langle 0,1,3,7,2,6\rangle$, 
               $\langle 0,1,3,7,2,12\rangle$, $\langle 0,1,3,7,2,18\rangle$, $\langle 0,1,3,7,6,2\rangle$, $\langle 0,1,3,7,6,8\rangle$, $\langle 0,1,3,7 6,22\rangle$, 
               $\langle 0,1,3,7,15,2\rangle$, $\langle 0,1,3,7,15,6\rangle$, $\langle 0,1,3,7,15,14\rangle$, $\langle 0,1,3,7,15,31\rangle$;}\\[-8mm] 
  \item $l=7$: {\footnotesize $\langle 0,1,3,2,4,9,25\rangle$, $\langle 0,1,3,2,4,15,9\rangle$, $\langle 0,1,3,2,4,15,14\rangle$, $\langle 0,1,3,2,4,15,31\rangle$, 
               $\langle 0,1,3,2,10,9,11\rangle$, $\langle 0,1,3,2,10,9,31\rangle$, $\langle 0,1,3,2,10,21,9\rangle$, $\langle 0,1,3,2,10,21,20\rangle$, $\langle 0,1,3,2,10,21,43\rangle$, 
               $\langle 0,1,3,7,2,6,25\rangle$, $\langle 0,1,3,7,2,12,6\rangle$, $\langle 0,1,3,7,2,12,17\rangle$, $\langle 0,1,3,7,2,12,31\rangle$, $\langle 0,1,3,7,2,18,6\rangle$, 
               $\langle 0,1,3,7,2,18,17\rangle$, $\langle 0,1,3,7,2,18,37\rangle$, $\langle 0,1,3,7,6,2,19\rangle$, $\langle 0,1,3,7,6,2,25\rangle$, $\langle 0,1,3,7,6,8,21\rangle$, 
               $\langle 0,1,3,7,6,8,31\rangle$, $\langle 0,1,3,7,6,22,2\rangle$, $\langle 0,1,3,7,6,22,21\rangle$, $\langle 0,1,3,7,6,22,45\rangle$, $\langle 0,1,3,7,15,2,6\rangle$, 
               $\langle 0,1,3,7,15,2,14\rangle$, $\langle 0,1,3,7,15,2,28\rangle$, $\langle 0,1,3,7,15,2,34\rangle$, $\langle 0,1,3,7,15,6,2\rangle$, $\langle 0,1,3,7,15,6,14\rangle$, 
               $\langle 0,1,3,7,15,6,24\rangle$, $\langle 0,1,3,7,15,6,38\rangle$, $\langle 0,1,3,7,15,14,2\rangle$, $\langle 0,1,3,7,15,14,6\rangle$, $\langle 0,1,3,7,15,14,16\rangle$, 
               $\langle 0,1,3,7,15,14,46\rangle$, $\langle 0,1,3,7,15,31,2\rangle$, $\langle 0,1,3,7,15,31,6\rangle$, $\langle 0,1,3,7,15,31,14\rangle$, $\langle 0,1,3,7,15,31,30\rangle$, 
               $\langle 0,1,3,7,15,31,63\rangle$.} 
\end{itemize}
A first step into an understanding of the possible solutions might be:
\begin{conjecture}
  If $\langle a_1',\dots,a_l'\rangle$ is a solution with $l\ge 2$, then $\langle a_1',\dots,a_{l-1}'\rangle$ is also a solution.
\end{conjecture}

We remark that we can generalize Equation~(\ref{eq_intricate_representation}) 
\begin{equation}
  N:=\sum_{i=1}^l \left(\frac{B}{2}\cdot B^{a_i}-B^{a_{i+1}+1}\right)\,+\, \frac{B}{2}.
\end{equation}
For all mentioned examples $\langle a_1',\dots,a_l'\rangle$ the square of the corresponding integer only consists of digits in base $B$ 
in $\{0,\tfrac{B}{2},\tfrac{B-2}{4}\}$ if $B\equiv 2\pmod 4$ and $B>2$. 
   
\begin{table}[htp!]
  \begin{center}
    \begin{tabular}{cccc}
      \hline
      pattern & condition & $D$ & explanation \\ 
      \hline
      $1(0)_n1$ & $n\ge 0$ & $\{0,1,2\}$ & Theorem~\ref{thm_square_digits_and_double_square_digits_everywhere} with $u=1$\\
      $1(0)_n1(0)_{n+1}1$ & $n\ge 0$ & $\{0,1,2\}$ & Theorem~\ref{thm_square_digits_and_double_square_digits_everywhere} with $u=1$\\
      $1(0)_{n+1}1(0)_n1$ & $n\ge 0$ & $\{0,1,2\}$ & Theorem~\ref{thm_square_digits_and_double_square_digits_everywhere} with $u=1$\\
      $1(0)_n11$ & $n\ge 1$ & $\{0,1,2\}$ & Theorem~\ref{thm_square_digits_and_double_square_digits_everywhere} with $u=1$\\
      $11(0)_n1$ & $n\ge 1$ & $\{0,1,2\}$ & Theorem~\ref{thm_square_digits_and_double_square_digits_everywhere} with $u=1$\\
      $1(0)_n2$ & $n\ge 0$ & $\{0,1,4\}$ & Lemma~\ref{lemma_r_s_construction} with $r=1$, $s=2$\\
      $2(0)_n1$ & $n\ge 0$ & $\{0,1,4\}$ & Lemma~\ref{lemma_r_s_construction} with $r=2$, $s=1$\\
      $102(0)_n201$ & $n\ge 2$ & $\{0,1,4\}$ & Lemma~\ref{lemma_1_2_construction_for_base_10} with $k=1$\\
      $201(0)_n102$ & $n\ge 2$ & $\{0,1,4\}$ & Lemma~\ref{lemma_1_2_construction_for_base_10} with $k=1$\\
      $1(0)_n202(0)_n1$ & $n\ge 2$ & $\{0,1,4\}$ & Lemma~\ref{lemma_1_2_construction_for_base_10_gen} (with $n=1$, $k=n$)\\ 
      $2(0)_n101(0)_n2$ & $n\ge 2$ & $\{0,1,4\}$ & Lemma~\ref{lemma_1_2_construction_for_base_10_gen} (with $n=1$, $k=n$)\\
      $1(0)_n3(0)_n1$ & $n\ge 1$ & $\{0,1,6\}$ & Lemma~\ref{lemma_two_blocks_of_zeros} with $\left(r_0,r_1,r_2;c\right)=(1,3,1;0)$\\
      $1(0)_n8(0)_n1$ & $n\ge 1$ & $\{0,1,6\}$ & Lemma~\ref{lemma_two_blocks_of_zeros} with $\left(r_0,r_1,r_2;c\right)=(1,8,1;0)$\\
      $4(0)_n127(0)_{n+2}4$ & $n\ge 2$ & $\{0,1,6\}$ & Lemma~\ref{lemma_two_blocks_of_zeros} with $\left(r_0,r_1,r_2;c\right)=(4,127,4;0)$\\
      $9(0)_n1$ & $n\ge 1$ & $\{0,1,8\}$ & Lemma~\ref{lemma_r_s_construction} with $r=B-1$, $s=1$\\
      $1(0)_n9$ & $n\ge 1$ & $\{0,1,8\}$ & Lemma~\ref{lemma_r_s_construction} with $r=1$, $s=B-1$\\
      $1(0)_n4(0)_n1$ & $n\ge 1$ & $\{0,1,8\}$ & Lemma~\ref{lemma_two_blocks_of_zeros} with $\left(r_0,r_1,r_2;c\right)=(1,4,1;0)$\\
      $1(0)_n9(0)_n1$ & $n\ge 1$ & $\{0,1,8\}$ & Lemma~\ref{lemma_two_blocks_of_zeros} with $\left(r_0,r_1,r_2;c\right)=(1,9,1;1)$\\
      $2(0)_n6(0)_n2$ & $n\ge 1$ & $\{0,2,4\}$ & Lemma~\ref{lemma_two_blocks_of_zeros} with $\left(r_0,r_1,r_2;c\right)=(2,6,2;1)$\\
      $5(0)_n5$ & $n\ge 1$ & $\{0,2,5\}$ & Theorem~\ref{thm_half_of_basis_as_digits} \\
      $5(0)_n505$ & $n\ge 2$ & $\{0,2,5\}$ & Theorem~\ref{thm_half_of_basis_as_digits} \\
      $505(0)_n5$ & $n\ge 3$ & $\{0,2,5\}$ & Theorem~\ref{thm_half_of_basis_as_digits} \\
      $15(0)_n85(0)_n5$ & $n\ge 2$ & $\{0,2,5\}$ & Lemma~\ref{lemma_two_blocks_of_zeros} with $\left(r_0,r_1,r_2;c\right)=(15,85,15;2)$\\
      $8(0)_n254(0)_{n+2}8$ & $n\ge 2$ & $\{0,4,6\}$ & Lemma~\ref{lemma_two_blocks_of_zeros} with $\left(r_0,r_1,r_2;c\right)=(8,254,8;0)$\\
      $2(0)_n2$ & $n\ge 0$ & $\{0,4,8\}$ & Theorem~\ref{thm_square_digits_and_double_square_digits_everywhere} with $u=2$\\
      $2(0)_n2(0)_{n+1}2$ & $n\ge 0$ & $\{0,4,8\}$ & Theorem~\ref{thm_square_digits_and_double_square_digits_everywhere} with $u=2$\\
      $2(0)_{n+1}2(0)_n2$ & $n\ge 0$ & $\{0,4,8\}$ & Theorem~\ref{thm_square_digits_and_double_square_digits_everywhere} with $u=2$\\
      $2(0)_n22$ & $n\ge 1$ & $\{0,4,8\}$ & Theorem~\ref{thm_square_digits_and_double_square_digits_everywhere} with $u=2$\\
      $22(0)_n2$ & $n\ge 1$ & $\{0,4,8\}$ & Theorem~\ref{thm_square_digits_and_double_square_digits_everywhere} with $u=2$\\
      $2(0)_n1(0)_n2$ & $n\ge 0$ & $\{0,4,9\}$ & Lemma~\ref{lemma_two_blocks_of_zeros} with $\left(r_0,r_1,r_2;c\right)=(2,1,2;0)$\\
      $(9)_n7$ & $\mathbf{n\ge 0}$ & $\{0,4,9\}$ & Section~\ref{subsec_infinite_non_zero}, Lemma~\ref{lemma_pattern_b_1_s}\\
      $(3)_n8$ & $n\ge 1$ & $\{1,2,4\}$ & Section~\ref{subsec_infinite_non_zero}\\
      $(3)_n5$ & $n\ge 0$ & $\{1,2,5\}$ & Section~\ref{subsec_infinite_non_zero}\\
      $123(3)_n5$ & $n\ge 0$ & $\{1,2,5\}$ & Section~\ref{subsec_infinite_non_zero}\\
      $(3)_n504485$ & $n\ge 4$ & $\{1,2,5\}$ & Section~\ref{subsec_infinite_non_zero}\\
      $(3)_n59$ & $n\ge 1$ & $\{1,2,8\}$ & Section~\ref{subsec_infinite_non_zero}\\
      $(3)_n4$ & $n\ge 0$ & $\{1,5,6\}$ & Section~\ref{subsec_infinite_non_zero}\\
      $(6)_n5$ & $n\ge 0$ & $\{2,4,5\}$ & Section~\ref{subsec_infinite_non_zero}\\
      $(6)_n515$ & $n\ge 1$ & $\{2,4,5\}$ & Section~\ref{subsec_infinite_non_zero}\\
      $2(3)_n5$ & $n\ge 1$ & $\{2,4,5\}$ & Section~\ref{subsec_infinite_non_zero}\\
      $(6)_n8$ & $n\ge 0$ & $\{2,4,6\}$ & Section~\ref{subsec_infinite_non_zero}\\
      $1(6)_n5$ & $n\ge 0$ & $\{2,5,7\}$ & Section~\ref{subsec_infinite_non_zero}\\
      $(6)_n7$ & $n\ge 0$ & $\{4,8,9\}$ & Section~\ref{subsec_infinite_non_zero}\\
      \hline
    \end{tabular}
    \caption{Explanations for the known infinite patterns in base $B=10$.}
    \label{table_explanations_b_10_patterns}  
  \end{center}
\end{table}

There are also some patterns that are quite similar but do not fit directly in the previously described family of patterns.
\begin{proposition}
  Let $B\equiv 2\pmod 4$ with $B\ge 4$ and
  \begin{equation}
    N= \frac{B}{2} + \frac{B}{2}\cdot B^2 + \frac{B}{2}\cdot B^3 + B^4\cdot\left(\frac{B}{2}\cdot B^3 - 1\right) + \frac{B}{2}\cdot B^{10} +
    \sum_{i=0}^l \frac{B}{2} \cdot B^{10+6\cdot 2^i}.
  \end{equation} 
  For each $l\in\mathbb{Z}$ we have $\varphi_B(N^2)=\left\{0,\frac{B}{2},\frac{B-2}{4}\right\}$.
\end{proposition}   
\begin{proof}
  Using the abbreviation $u=\tfrac{B}{2}$ we compute
  \begin{eqnarray*}
    N^2 &=& u^2+u\cdot B^3+\left(u^2-u\right)\cdot B^4 +\left(u^2-u\right)\cdot B^6 +u\cdot B^8+u\cdot B^{10} +u\cdot B^{13}+\left(u^2-u\right)\cdot B^{14}\\ 
        && +u\cdot B^{18}+u^2\cdot B^{20} +\sum_{i=0}^l u^2 B^{20+12\cdot 2^i} +\sum_{0\le i<j\le l} u\cdot B^{21+6\cdot 2^i+6\cdot 2^j} \\ 
        && +\sum_{i=0}^l \left( u\cdot B^{11+6\cdot 2^i} + u\cdot B^{13+6\cdot 2^i} -u\cdot B^{14+6\cdot 2^i} + u\cdot B^{18+6\cdot 2^i}  + u\cdot B^{21+6\cdot 2^i}\right) \\ 
        &=& u^2+u\cdot B^3+\left(u^2-u\right)\cdot B^4 +\left(u^2-u\right)\cdot B^6 +u\cdot B^8+u\cdot B^{10} +u\cdot B^{13}+\left(u^2-u\right)\cdot B^{14}\\ 
        && +u\cdot B^{18}+u^2\cdot B^{20+6\cdot 2^{l+1}} +\sum_{0\le i< j\le l} u\cdot B^{21+6\cdot 2^i+6\cdot 2^j}\\
        && +\sum_{i=0}^l \left( u\cdot B^{11+6\cdot 2^i} + u\cdot B^{13+6\cdot 2^i} +\left(u^2-u\right)\cdot B^{14+6\cdot 2^i} + u\cdot B^{18+6\cdot 2^i}  + u\cdot B^{21+6\cdot 2^i}\right). 
  \end{eqnarray*}
  Noting that $u=\tfrac{B}{2}$, $u^2-u=\tfrac{B-2}{4}\cdot B$, and $u^2=\tfrac{B-2}{4}\cdot B+\tfrac{B}{2}$, it remains to observe that the exponents 
  $0$, $1$, $3$, $4$, $5$, $6$, $7$, $8$, $10$, $13$, $14$, $15$, $18$, $\{20,21\}+6\cdot 2^{l+1}$, $\left\{11,13,14,15,18,21\right\}+6\cdot 2^i$ for $0\le i\le l$, and 
  $21+6\cdot 2^i+6\cdot 2^j$ for $0\le i<j\le l$ of the occurring non-zero digits are pairwise different. 
\end{proof} 
Examples for $B=10$ are given by 
{\footnotesize\begin{eqnarray*}
  50049995505^2 &=& 2505002050050520205025,\\
  50000050049995505^2 &=& 2500005005002055502050050520205025,\text{ and }\\
  50000000000050000050049995505^2 &=& 2500000000005000005005002050505005002055502050050520205025. 
\end{eqnarray*}} 
The following three examples might also be part of an infinite pattern:
{\footnotesize\begin{eqnarray*} 
50000004999999955050005^2 &=& 2500000500000020505000050500052020502050500025,\\
50050000004999999999955050005^2 &=& 2505002500500500000020500505500050500050002020502050500025,\text{ and}\\
50000005004999999999955050005^2 &=& 2500000500500025050020505000050050550050002020502050500025.
\end{eqnarray*}}

\subsubsection{Summary}
\label{subsec_summary_infinite_patterns}
The idea of this subsection is to summarize our known knowledge on the infinite patterns for base $B=10$. In Table~\ref{table_explanations_b_10_patterns} we list 
the known patterns mentioned at \url{www.worldofnumbers.com/threedigits.htm}, excluding those that are discussed in Subsection~\ref{subsec_intricate_patterns}, and 
a few new ones. In all cases we give a reference to a more general result. Theorem~\ref{thm_square_digits_and_double_square_digits_everywhere} with $u\in\{1,2\}$, 
Theorem~\ref{thm_half_of_basis_as_digits}, and Lemma~\ref{lemma_many_ones_1_four} give further patterns not mentioned in the table. Due to the relation between 
examples with decimal digits in $\{0,1,2\}$ and those with decimal digits in $\{0,4,8\}$, see Lemma~\ref{lemma_multiplied_triples}, we have dropped the patterns 
for the latter.    

\subsubsection{Sporadic solutions}
\label{subsec_sporadic_solutions}
Let $D\subseteq\{0,\dots,9\}$ be a set of three decimal digits. All integers $N$ with $D_{10}(N^2)=D$ that are not contained in one of the constructions of the previous subsections 
are called \emph{sporadic solutions}. Of course it can happen that a new general construction is discovered so that some sporadic solutions may lose this label in the future.  
Hisanori Mishima has tabulated all sporadic solutions for the case $0\in D$ for all $N\le 10^{23}$ and for the case $0\notin D$ for all $N\le 10^{25}$, see 
\url{www.asahi-net.or.jp/~KC2H-MSM/mathland/math02/math0210.htm}. We state a few newly discovered sporadic solutions in Table~\ref{table_sporadic_solutions}. 

\begin{table}[htp]
  \begin{center}
    \begin{tabular}{ll}
      \hline
      $N$ & $\{D_{10}(N)$\\
      \hline
      471287714788971663493899       & $\{0,1,2\}$ \\
      10000009999995510010001000001  & $\{0,1,2\}$ \\
      77470059130002034719700749     & $\{0,1,6\}$ \\
      205524700326856587391168       & $\{0,2,4\}$ \\
      447241797269721007814765       & $\{0,2,5\}$ \\
      50050000004999999999955050005  & $\{0,2,5\}$ \\
      50000005004999999999955050005  & $\{0,2,5\}$ \\
      141422082876067219949805050005 & $\{0,2,5\}$ \\
      21214250022106461574572502     & $\{0,4,5\}$ \\
      2108436491907081488939581538   & $\{0,4,5\}$ \\
      942575429577943326987798       & $\{0,4,8\}$ \\
      26395073915340646948470264     & $\{2,6,9\}$ \\ 
      25781108305591628417975738     & $\{4,5,6\}$ \\
       \hline
    \end{tabular}
    \caption{Sporadic solutions $N$ with $\left|D_{10}(N^2)\right|=3$ and many digits}
    \label{table_sporadic_solutions}   
  \end{center}
\end{table}  

\subsection{Feasible and infeasible triples of digits}
\label{sec_triples}
Given a base $B\ge 2$ and a set $D\subseteq \{0,1,\dots,B-1\}$ of cardinality $3$, does there exist a positive integer $N$, not divisible by $B$, such that $D_B(N^2)=D$? For 
the positive answer, where we call $D$ \emph{feasible}, it is sufficient to provide an example. For the negative answer, where we call $D$ \emph{infeasible}, we have the 
following rather easy kind of certificate, which we want to explain in terms of an example. Consider the set
\begin{eqnarray*}
  && \{n^2\operatorname{rem} 100\,:\, 1\le n\le 99, n\not\equiv 0\pmod {10}\}\\ &=&\{1,4,9,16,21,24,25,29,36,41,44,49,56,61,64,69,76,81,84,89,96\}.
\end{eqnarray*}
With this we can easily check that e.g.\ $D=\{0,2,8\}$ is infeasible since none of the nine combinations of two digits in $D$ occurs in the above set. Of course, we can also 
consider the last three or last four digits of $N^2$. For bases $B\in \{9,10\}$ we remark that the last two digits are sufficient, i.e., the infeasible 
patterns are characterized by the non-existence of $2$-suffixes, see Section~\ref{sec_suffices}. 

As a more general statement we observe that each $D$ not containing an even number is infeasible, i.e., not all digits can be odd.
\begin{lemma}
  Let $B\in\mathbb{N}_{\ge 2}$ and $N\in\mathbb{N}$ such that $l_B(N^2)\ge 2$. If $B\equiv 2\pmod 4$, then $D_B(N^2)$ contains an even element, i.e., $N^2$ contains an 
  even digit in base $B$-representation.
\end{lemma} 
\begin{proof}
   Let us write $[a_1,a_0]_B$ for the last two digits of $N$ and $[b1,b_0]_B$ for the last two digits of $N^2$. If $2|N$, then $2|b_0$, i.e., $b_0$ is even. Thus, 
   we assume $a_0=2k+1$ for some integer $0\le k\le B/2-1$. Now we consider
   $$
     b_1\cdot B+b_0\equiv 2a_0a_1\cdot B+a_0^2\equiv 2a_0a_1\cdot B +4k^2+4k+1 \pmod{B^2}, 
   $$  
   which implies 
   $$
     b_1\equiv 2a_0a_1 +\left(4k^2+4k+1\right)\operatorname{div}B\pmod B.
   $$
   Since $B$ is even we have $\left(4k^2+4k+1\right)\operatorname{div}B=\left(4k(k+1)\right)\operatorname{div}B$, so that $2a_0a_1\equiv 0\pmod 2$ and $B\not\equiv 0\pmod 4$ imply 
   $b_1\equiv 0\pmod 2$, i.e., $b_1$ is even. 
\end{proof}
 
Surely, there are ${B\choose 3}$ possible triples for base $B$. The two undecided triples for $B=10$ are given by $\{0,1,3\}$ and $\{6,7,8\}$. We remark that for the 
triple $\{5,6,8\}$ the only known solution $816^2=665856$ is rather small. For base $B=9$ there are $84$ triples, where $20$ triples are infeasible. The remaining 
$64$ triples allow solutions. Moreover, for $8$ triples we have found infinite patterns, see Appendix~\ref{sec_b_9} for the details.

\subsection{Suffixes and prefixes}
\label{sec_suffices}
For a given base $B$ and \emph{$n$-suffix} of an integer $M$ consists of the last $n$ digits of $M$ in base $B$ representation. Given a subset $D\subseteq \{0,1,\dots,B-1\}$ 
of the possible digits, we denote by $\omega_B^D(n)$ the number of positive integers $x$ with at least $n$ digits such that the last $n$ digits of the corresponding square are 
contained in $D$. Let us have an example for $B=10$ and $D=\{7,8,9\}$. Since neither $7$ nor $8$ is a quadratic residue modulo $10$, the last digit of the square has to be a $9$, 
so that the last digit of the root is either a $3$ or a $7$, i.e., $\omega_{10}^{\{7,8,9\}}(1)=2$. Using the notation and conditions from Lemma~\ref{lemma_mult_last_digits} we write
$$
  \big([a_{n-1},\dots,a_0]_B\big)^2=[c_{n-1} | b_{n-1},\dots,b_0]_B
$$  
in order to indicate that the square $M^2$ of the integer $M=[a_{n-1},\dots,a_0]_B$ ends with $[b_{n-1},\dots,b_0]_B$ and has $c_{n-1}$ as its next carryover. E.g.\ we have 
$[3]_{10}^2=[0|9]_{10}$ and $[7]_{10}^2=[4|9]_{10}$. Next we want to extend these $1$-prefixes to $2$-prefixes such that the last two digits of their corresponding squares are 
contained in $D=\{7,8,9\}$, i.e., we want to solve $[x,3]_{10}^2=[\star| y,9]_{10}$ and $[x,7]_{10}^2=[\star |y,9]_{10}$ for suitable $x\in\{0,\dots,9\}$ and $y\in \{7,8,9\}$, where 
$\star$ denotes an arbitrary (non-negative) integer. Using Lemma~\ref{lemma_mult_last_digits} we can rewrite the two equations to
$$
  2\cdot 3\cdot x+0=6x+0\equiv y\in\{7,8,9\} \pmod {10}\quad\text{and}\quad 
  14x+4\equiv y\in\{7,8,9\} \pmod {10}.
$$ 
Multiplying the equations with $7$ and $3$, respectively, and using the $\operatorname{rem}$ operator instead of $\equiv$, we can reformulate to
$$
  2x\operatorname{rem} 10\in\{3,6,9\}\quad\text{and}\quad 
  2x+2\operatorname{rem} 10\in\{1,4,7\}\,\Leftrightarrow\, 2x\operatorname{rem} 10\in\{2,5,9\}. 
$$
Since $2$ divides $10$ we cannot further simplify these equations. Moreover, $x$ is a solution iff $x+10/2=x+5$ is a solution. In the first case only $6$ is divisible by $2$ and in the 
second case only $2$ is divisible by $2$, so that $x\in\{3,8\}$ in the first case and $x\in\{1,6\}$ in the second case. I.e., we have 
$[33]_{10}^2=[1|89]_{10}$, $[83]_{10}^2=[4|89]_{10}$, $[17]_{10}^2=[1|89]_{10}$, $[67]_{10}^2=[8|89]_{10}$, and $\omega_{10}^{\{7,8,9\}}(2)=4$. 
Similar computations for the next digit gives
\begin{itemize}
  \item $[x33]_{10}$: $6x+10\operatorname{rem}10\in\{7,8,9\}\,\Longleftrightarrow\, 2x\operatorname{rem}10\in\{3,6,9\}\,\Rightarrow x\in\{3,8\}$;
  \item $[x83]_{10}$: $6x+68\operatorname{rem}10\in\{7,8,9\}\,\Longleftrightarrow\, 2x\operatorname{rem}10\in\{0,3,7\}\,\Rightarrow x\in\{0,5\}$;
  \item $[x17]_{10}$: $14x+2\operatorname{rem}10\in\{7,8,9\}\,\Longleftrightarrow\, 2x\operatorname{rem}10\in\{1,5,8\}\,\Rightarrow x\in\{4,9\}$;
  \item $[x67]_{10}$: $14x+44\operatorname{rem}10\in\{7,8,9\}\,\Longleftrightarrow\, 2x\operatorname{rem}10\in\{2,5,9\}\,\Rightarrow x\in\{1,6\}$;
\end{itemize}  
and $\omega_{10}^{\{7,8,9\}}(3)=8$. Thus, we obtain $[333]_{10}^2=[2|889]_{10}$, $[833]_{10}^2=[5|889]_{10}$, $[083]_{10}^2=[6|889]_{10}$, $[583]_{10}^2=[9|889]_{10}$, 
$[417]_{10}^2=[5|889]_{10}$, $[917]_{10}^2=[12|889]_{10}$, $[167]_{10}^2=[5|889]_{10}$, and $[667]_{10}^2=[12|889]_{10}$. For general $B$, $D$ and fixed last digit $a_0$ 
of the root the corresponding equations always have the form
\begin{equation}
  \label{eq_compute_next_matching_digit}
  \left(2a_0x + \gamma\right) \operatorname{rem} B\in D
\end{equation}
for some suitable integer $\gamma$ depending on the other data. If $2a_0$ is invertible in $\mathbb{Z}_B:=\mathbb{Z}/B\mathbb{Z}$, then there are exactly $|D|$ solutions for 
$x$. In our situation $a_0$ cannot be invertible in $\mathbb{Z}_{10}$ since $2$ is zero divisor. However both choices $a_0\in\{3,7\}$ are invertible in $\mathbb{Z}_{10}$, so that 
the number of solutions is given by $2\cdot|\{d\in D\,:\, d\equiv \gamma\pmod 2\}|$, i.e., there are two solutions if $\gamma$ is even and four solutions if $\gamma$ is odd. So far 
only the first case occurred and the last digits of the square are uniquely determined. This changes when we extend the next digit and determine 
$\omega_{10}^{\{7,8,9\}}(4)=24$. 
Computing the next values $\omega_{10}^{\{7,8,9\}}(n)$ with a computer program suggests:
\begin{conjecture}
  \label{conjecture_789}
  For all $n\ge 3$ we have $\omega_{10}^{\{7,8,9\}}(n)=8\cdot 3^{n-3}$.
\end{conjecture}
In Table~\ref{table_conjectured_formulas_num_suffix} we list a few more such observations. 
Of course we may write for cases like $D=\{1,3,9\}$, where no $2$-suffixes or no $3$-suffixes exist at all, $\omega_{10}^D=0\cdot 3^{n-3}$ for $n\ge 3$. 
However, the existence of an integer $k$ and a constant $c\in\mathbb{R}_{>0}$, depending on $B$ and $D\subseteq \{0,1,\dots,B-1\}$ with $|D|\ge 2$, such that 
$\omega_B^D(n)=c\cdot |D|^{n-k}$ for all $n\ge k$ seems to be wrong in general. For e.g.\ $B=10$ and $D=\{0,1,6\}$ we have 
\begin{eqnarray*}
  \left(\omega_{B}^D \right)_{n=1,\dots,20} &=& 4, 12, 36, 112, 336, 1000, 3000, 8984, 26952, 80984, 243016, 728856, 2186568,\\ 
  &&  6560856, 19682824, 59047960, 177146440, 531451096, 1594341000, 4783037336.          
\end{eqnarray*}

\begin{table}[htp]
  \begin{center}
    \begin{tabular}{ll}
      \hline 
      conjectured pattern & cases $d_1d_2d_3$\\ 
      \hline
      $16\cdot 3^{n-3}$, $n\ge3$ & 019,123,127,189,239,279\\ 
      $24\cdot 3^{n-3}$, $n\ge2$ & 012, 029,128,289\\
      $24\cdot 3^{n-3}$, $n\ge3$ & 013,015,017,039,059,079,138,158,178,389,589,789\\  
      $32\cdot 3^{n-3}$, $n\ge3$ & 018,089,129\\
      $40\cdot 3^{n-3}$, $n\ge3$ & 235,257,258\\
      $48\cdot 3^{n-3}$, $n\ge4$ & 048\\ 
      $56\cdot 3^{n-3}$, $n\ge3$ & 125,259\\
      $80\cdot 3^{n-3}$, $n\ge3$ & 025\\
      none or involved & 014,016,024,034,036,045,046,047,049,056,067,069,124,126,134,136,\\ 
                       & 145,146,147,148,149,156,167,168,169,234,236,245,246,247,248,249,\\ 
                       & 256,267,269,345,346,347,348,349,356,367,368,369,456,457,458,459 \\
                       & 467,468,469,478,479,489,567,569,678,679,689\\ 
      \hline
    \end{tabular}
    \caption{Conjectured formulas for $\omega_{10}^{\{d_1,d_2,d_3\}}(n)$} 
    \label{table_conjectured_formulas_num_suffix} 
  \end{center}
\end{table}
\begin{proposition}
  \label{prop_suffix}
  If $B$ is prime strictly larger than $2$ and $D\subseteq \{0,1,\dots,B-1\}$ with $|D|\ge 2$, then we have $\omega_{B}^{D}(n)=|D|^{n-1}\cdot \omega_{B}^{D}(1)$ for all $n\ge 1$. 
\end{proposition} 
\begin{proof}
  Using Equation~(\ref{eq_compute_next_matching_digit}) it remains to observe that $2a_0$ is invertible in $\mathbb{Z}_B$ since $a_0\neq 0$ ($a_0\not\equiv 0\pmod B$) and 
  $\mathbb{Z}_B$ is a (finite) field.
\end{proof}
So, if $\omega_B^D(1)\neq 0$, which is just a question about the number of quadratic residues in $D$, then $\omega_B^D(n)= c\cdot |D|^n$ for some constant $c\in\mathbb{R}_{>0}$, 
i.e., asymptotically we have exponential behavior.
\begin{conjecture}
  \label{conjecture_suffix_asymptotic}
  If $B\in\mathbb{N}_{\ge 2}$ and $D\subseteq \{0,1,\dots,B-1\}$ with $|D|\ge 2$, then there exists a constant $c\in\mathbb{R}_{>0}$ such that
  $\lim\limits_{n\to\infty} \omega_B^D(n) /|D|^{n}=c$.  
\end{conjecture} 
In other words, for subsets $D$ of cardinality $3$ we conjecture that for each $B$ and each $D$ there exists a constant $c_{B,D}\in\mathbb{R}_{\ge 0}$ such that 
$\omega_B^D(n)\simeq c\cdot 3^n$. In Table~\ref{table_conjectured_formulas_num_suffix} we can see that the value for $c$ can differ significantly for different sets $D$, $D'$ of 
allowed digits even if $\#D=\# D'$ and both sets allow $n$-suffixes for sufficiently large values of $n$.

Proposition~\ref{prop_suffix} can be partially generalized. If $x$ denotes the new digit, i.e., $a_{2k}$ or $a_{2k+1}$ in the following, then we have
\begin{eqnarray}
  B\cdot c_{2k}+b_{2k} &=& 2a_0x +2\sum_{i=1}^{k-1}a_i a_{2k-i}\,+\,a_k^2+c_{2k-1}\\
  B\cdot c_{2k+1}+b_{2k+1} &=& 2a_0x +2\sum_{i=1}^{k}a_i a_{2k+1-i}\,+\,c_{2k} 
\end{eqnarray}
for $k\ge 1$. By $\widetilde{\omega}_B^D$ we denote the number of solutions counted by $\omega_B^D$ satisfying the additional condition that $a_0$ is invertible in 
$\mathbb{Z}/B\mathbb{Z}$. 
 
\begin{lemma}Let $B\equiv 2\pmod 4$, $k\ge 1$, and $D\subseteq\{0,1,\dots,B-1\}$.
  \begin{itemize}
    \item[(i)]    The number of solutions with $x\equiv 0\pmod 2$ and $a_0$ odd equals the number of solutions with $x\equiv 1\pmod 2$ and $a_0$ odd.
    \item[(ii)]   The number of solutions with $x\equiv 0\pmod 2$ and $a_0$ invertible in $\mathbb{Z}/B\mathbb{Z}$ equals the number of solutions with $x\equiv 1\pmod 2$ 
                  and $a_0$ invertible in $\mathbb{Z}/B\mathbb{Z}$.
    \item[(iii)]  The number of solutions with $c_n\equiv 0\pmod 2$ and $a_0$ odd equals the number of solutions with $c_n\equiv 1\pmod 2$ and $a_0$ odd, where $n\in\{2k,2k+1\}\backslash\{0\}$.
    \item[(iv)]   The number of solutions with $c_n\equiv 0\pmod 2$ and $a_0$ invertible in $\mathbb{Z}/B\mathbb{Z}$ equals the number of solutions with $c_n\equiv 1\pmod 2$ and 
                  $\mathbb{Z}/B\mathbb{Z}$, where $n\in\{2k,2k+1\}\backslash\{0\}$.
    \item[(v)]    $\widetilde{\omega}_{B}^{D}(2k+2)=|D|\cdot \widetilde{\omega}_{B}^{D}(2k+1)$ for all $k\ge 1$.
    \item[(vi)]   The number $\widetilde{\omega}_{B}^{D}(2k)$ splits equally into the four cases $\left(a_k\!\!\mod 2,c_{2k-1}\!\!\mod 2\right)\in \{0,1\}^2$ for all $k\ge 2$.
    \item[(vii)]  $\widetilde{\omega}_{B}^{D}(2k+1)=|D|\cdot \widetilde{\omega}_{B}^{D}(2k)$ for all $k\ge 2$. 
    
    If $\widetilde{\omega}_{B}^{D}(2k)$ splits evenly into $a_k+c_{2k-1}\equiv 0\pmod 2$ and $a_k+c_{2k-1}\equiv 1\pmod 2$, 
                then $\widetilde{\omega}_{B}^{D}(2k+1)=|D|\cdot \widetilde{\omega}_{B}^{D}(2k)$.                
  \end{itemize}
\end{lemma}
\begin{proof}
  If $a_0$ is invertible in $\mathbb{Z}/B\mathbb{Z}$, then $B\equiv 0\pmod 2$ implies that $a_0$ is odd. So, if $x$ is a solution, so is $x+B/2$. This already gives (i), (ii), (iii), 
  and (iv) since $B/2\equiv 1\pmod 2$ and $2a_0\cdot B/2\equiv B\pmod{2B}$ using that $a_0$ is odd. For (v) note  that $\omega_{B}^{D}(2k+1)$ splits evenly into 
  $c_{2k}\equiv 0\pmod 2$ and $c_{2k}\equiv 1\pmod 2$ if $k>0$. Since $a_0$ is invertible in $\mathbb{Z}/B\mathbb{Z}$, for each $b_{2k+1}\in D$ we have 
  $2\cdot \widetilde{\omega}_{B}^{D}(2k+1)/2$ solutions. For (vi) we consider the situation when the digit $a_k$ is determined by a solution $x$. From (ii) we conclude that 
  those cases split equally onto $a_k\!\!\mod 2\in\{0,1\}$. Now we further have to extend those $k$-digit numbers to $2k$-digit numbers. During this process the digit $a_k$ does 
  not change any more, while the parities of the subsequent carryovers split half-half due to (iv). Applied iteratively, we end up with statement (vi). For 
  (vii) we observe that (vi) implies that $\widetilde{\omega}_{B}^{D}(2k)$ splits evenly into $a_k+c_{2k-1}\equiv 0\pmod 2$ and $a_k+c_{2k-1}\equiv 1\pmod 2$, so that we can proceed 
  as for (v).   
\end{proof}
Observe that $[3]_{10}^2=[0|9]_{10}$ and $[7]^2_{10}=[4|9]_10$, i.e., for $c_1\equiv 0\pmod 2$ there are two and for $c_1\equiv 1\pmod 2$ there is none possibility. Additionally 
$a_1$ is odd in both cases and $\omega_{10}^{\{7,8,9\}}(3)=2\cdot \omega_{10}^{\{7,8,9\}}(2)$. So, statements (i), (ii), (iii), and (iv) are  valid for $n\ge 1$ only. Also the 
condition that $a_0$ is odd is crucial. For $D=\{0,4,8\}$ consider e.g.\ the {\lq\lq}pair{\rq\rq} of solutions $002$ and $502$ with $3$ digits. We have $[002]_{10}^2=[0|004]_{10}$ 
and $[502]_{10}^2=[2|004]_{10}$, i.e., both carryovers are even. And indeed, those two solutions with $3$ digits generate $12$ solutions with $4$ digits and not $2\cdot 3=6$ as 
the argument in the proof of (v) would suggest. Indeed $\omega_{10}^{\{0,4,8\}}(4)=144 \neq 3\cdot 36=3\cdot \omega_{10}^{\{0,4,8\}}(3)$.       

\begin{corollary}
  \label{cor_omega_tilde}
  Let $B\equiv 2\pmod 4$, $k\ge 1$, and $D\subseteq\{0,1,\dots,B-1\}$. Then, for each $n\ge 3$ we have 
  $\widetilde{\omega}_{B}^{D}(n)=|D|^{n-3}\cdot \widetilde{\omega}_{B}^{D}(3)$. 
\end{corollary}

Let us denote the number of solutions $x\in\{1,\dots,B-1\}$ with $x^2\equiv a\pmod B$ by $s_B(a)$.\footnote{Using the Chinese remainder theorem we can reduce the problem to prime powers 
and then to primes $B$. For odd primes the Legendre-symbol and the quadratic reciprocity law comes into play.} Clearly we have
\begin{equation} 
  \omega_B^D(1)=\sum_{d\in D} S_B(d),
\end{equation}
where $s_B(0)=0$ for all $B\ge 2$. Note that we have $\widetilde{\omega}_B^D(n)=\omega_B^D(n)$ for all $n\ge 1$ iff $\widetilde{\omega}_B^D(1)=\omega_B^D(1)$. Let us have a closer look 
at the special case $B=10$. Here the non-zero quadratic residues are given by $\{1, 4, 5, 6, 9\}$. So, if e.g.\ $D\cap \{1,4,5,6,9\}=\emptyset$, then $\omega_{10}^D(n)=0$ for all 
$n\ge 1$. Corollary~\ref{cor_omega_tilde} implies:
\begin{proposition}  
  If $B=10$ and $D\subseteq\{0,1,\dots,B-1\}$ with $D\cap\{4,5,6\}=\emptyset$, then we have $\widetilde{\omega}_B^D(n)=\omega_B^D(n)$ for all $n\ge 1$ and 
  $\omega_{B}^{D}(n)=|D|^{n-3}\cdot \omega_{B}^{D}(3)$ for all $n\ge 3$. 
\end{proposition}
So, especially Conjecture~\ref{conjecture_789} and several other cases of Table~\ref{table_conjectured_formulas_num_suffix} are true. Lemma~\ref{lemma_multiplied_triples} 
indicates that $\omega_{10}^{\{0,1,2\}}(n)$ is closely related with $\omega_{10}^{\{0,4,8\}}(n)$ while the number of digits can change when going from $N$ to $2N$.  

\medskip

Our next aim is to define a counting function $\alpha_B^D(n)$ for $n$-prefixes similar to $\omega_B^D$ for $n$-suffixes. Before we give a definition of an $n$-prefix, we study 
a few examples. Suppose we want the $2$-prefix 
to be $41$. If only $4$-digit squares are allowed, then 
$$
  \sqrt{4100}=64.03124237\quad\text{and}\quad\sqrt{4199}=64.79969136
$$  
shows that there is no such integer. However, is we allow $5$-digit squares, then
$$
  \sqrt{041000}=202.4845673\quad\text{and}\quad \sqrt{041999}=204.9365756
$$ 
shows that we have the following two solutions
$$
  203^2=41209\quad\text{and}\quad 204^2=41616
$$
both having $[4,1]$ as their leading two digits.

Lets us have another example and search for a number whose square starts with $569569$. From 
$$
  \sqrt{569569000000}=754697.9528
  \quad\text{and}\quad
  \sqrt{569569999999}=754698.6153  
$$
we conclude that $754698$ is a solution, i.e., $754698^2=569569071204$ starts with $569569$. 
Similarly, 
$$
  \sqrt{05695690000000}=2386564.476
  \quad\text{and}\quad
  \sqrt{05695699999999}=2386566.571  
$$
gives us the solutions
$$
  2386565^2=5695692499225\quad\text{and}\quad
  2386566^2=5695697272356.
$$
In order to try to explain a few things we observe that 
\begin{equation}
  \label{eq_root_diff}
  \sqrt{x\cdot C+(C-1)}-\sqrt{x\cdot C}
\end{equation}
is monotonically decreasing in $x$ for each $C>0$. We will choose $C=B^k$, where $B$ is the basis, and will assume suitable bounds for $x$. 
Next we observe
\begin{equation}
  \sqrt{s+t}-\sqrt{s}=\frac{t}{\sqrt{s+t}+\sqrt{t}}
\end{equation}
for all $s,t>0$. Not that the denominator of the right hand side is bounded between $\max\{2\sqrt{t},\sqrt{s}\}$ and $2\sqrt{s+t}$.

For our first cases, i.e.\ squares without a {\lq\lq}leading zero{\rq\rq} we have to choose $C=B^k$ for some positive integer $k$ and 
$B^{k-1}\le x\le B^k-1$. With this we obtain
\begin{equation}
   \frac{1}{2}\le \sqrt{x\cdot C+(C-1)}-\sqrt{x\cdot C} 	< \sqrt{B},
\end{equation} 
i.e., we cannot guarantee a solution but there may be several solutions. For the second case we have to choose $C=B^{k+1}$ for some positive integer $k$ and 
$B^{k-1}\le x\le B^{k}-1$. With this we obtain
\begin{equation}
   \frac{\sqrt{B}}{2}\le \sqrt{x\cdot C+(C-1)}-\sqrt{x\cdot C} < B,
\end{equation}  
i.e. for $B\ge 5$ we can always guarantee a solution, while the total number of solutions is upper bounded by $B$. Note that these estimations are valid for every 
sequence of leading digits, i.e., no special properties are used. So, if $0\notin D$, then there are exactly $|D|^n$ squares and if $0\in D$ there are exactly 
$(|D|-1)\cdot |D|^{n-1}$ squares with the leading digits in $D$ (including the possibility of a leading $0$), provided that $B$ is sufficiently large. 
\begin{definition}
  By $\alpha_B^D$ we denote the number of integers $B^{n-\tfrac{1}{2}}\le x<B^{n+\tfrac{1}{2}}$ such that the first $n$ digits in base $B$ of $x^2$ are contained in $D$.
\end{definition}
From the above we conclude 
\begin{equation}
  \left\lfloor\frac{\sqrt{B}}{2}\right\rfloor \cdot |D|^n \le \alpha_B^D(n)\le \left(\left\lceil\frac{\sqrt{B}}{2}\right\rceil+B\right)\cdot |D|^n. 
\end{equation}
\begin{conjecture}
  \label{conjecture_prefix_asymptotic}
  If $B\in\mathbb{N}_{\ge 2}$ and $D\subseteq \{0,1,\dots,B-1\}$ with $|D|\ge 2$, then there exists a constant $c\in\mathbb{R}_{>0}$ such that
  $\lim\limits_{n\to\infty} \alpha_B^D(n) /|D|^{n}=c$.  
\end{conjecture} 

\subsection{Search algorithms}
\label{sec_search_algorithms}

Suppose that $P_{B,D}(n)$ is a list of all numbers $x<B^n$ such that the last $n$ digits of $x^2$ in base $B$ are contained in $D$. We can easily compute 
$P_{B,D}(n+1)$ from $P_{B,D}(n)$ by prepending another suitable digit. As shown is Section~\ref{sec_suffices} the possibilities for the new digit do not need to be tried 
but can be computed. If we utilize the recursion from Lemma~\ref{lemma_mult_last_digits} and also store the carryovers, then $P_{B,D}(n)$ 
can be computed in $\# P_{B,D}(n)\cdot O(n)$ operations, which equals $O\!\left(n\cdot |D|^n\right)$ operations if Conjecture~\ref{conjecture_suffix_asymptotic} is true.    
If we use depth-first search than be space requirement is just $O(n)$. We may also compute all $n$-digit number in base $B$ such that there squares consist of digits in $D$ only. 
To this end we have just compute the leading $n$ digits for all elements of $P_{B,D}(n)$ using the recursion from Lemma~\ref{lemma_mult_last_digits}. Since the number 
of candidates while decrease asymptotically, the algorithm will still need $\# P_{B,D}(n)\cdot O(n)$ operations only. 

For exhaustive lists of $n$-prefixes we just have to loop over all $|D|^n$ choices for the leading $n$ digits of the square and compute bounds for the square roots. From 
Inequality~(\ref{conjecture_prefix_asymptotic}) we conclude that $n\cdot |D|^n$ operations are needed. Again, if we use depth-first search than be space requirement is just $O(n)$.
 
Since the conjectures order of magnitude of $\omega_B^D(n)$ as well as $\alpha_B^D(n)$ is $O(|D|^n)$, computing the first $s$ digits using prefixed and the last $n-s$ digits using suffixes, 
we end up with $O(n\cdot |D|^n)$ operations, i.e., $O(n\cdot 3^n)$ operations in our situation where $\#D=3$. Slightly more sophisticated, we can compute 
lists of $p$-digit prefixes and $s$-digit suffixes, so that they overlap in $q$ digits, see e.g.\ \url{http://djm.cc/rpa-output/arithmetic/digits/squares/three.digits.s}. Sticking 
to $\#D=3$ and assuming the truth of Conjecture~\ref{conjecture_suffix_asymptotic}, the computation of the two lists need $O(n\cdot 3^p)$ and $O(n\cdot 3^s)$ operations, respectively. 
If we assume that the overlaps scatter almost evenly over the $10^q$ possible buckets for $B=10$, then combining all possibilities needs $O(n\cdot 3^{p+s}/10^q)$ operations. 
Choosing $p=s=\ln 10/(2\ln 10-\ln 3)\cdot n$ and $q=\ln(3)/\ln(10)\cdot n$ yields a search algorithm that needs roughly $O(n\cdot 3^{0.657n})$ operations. The big drawback is that the 
space requirements are of the same order of magnitude. However, if we compute lists of $k$-suffixes, $k$-prefixes, loop over all $O(3^{2k})$ combinations, and apply the 
just described algorithm for the central $n$ digits, then we end up with a search algorithm using $O((n+k)\cdot 3^{2k+0.657n})$ operations and $O(n\cdot 3^{0.657n})$ space 
for checking all numbers with at most $n+2k$ digits.

\begin{table}[htp!]
  \begin{center}
  \vspace*{-1.7cm}
  {\tiny
    \begin{tabular}{rrrrrrrr}
\hline
$D$ & $c_1$ & $c_2$ & $\#$ $25$-digit & est.\ $20$-digit & est.\ $>25$-digit & est.\ $10-25$-digit & $\#$ $10-25$-digit \\ 
\hline  
$\mathbf{\{0,1,2\}}$ & 0.89 & 1.79 & 8 & 0.19 & 1.03 & 4.52 & 4 \\ 
$\{0,1,3\}$ & 0.30 & 2.08 & 0 & 0.08 & 0.40 & 1.75 & 0 \\
$\mathbf{\{0,1,4\}}$ & 2.10 & 2.17 & 25 & 0.55 & 2.94 & 12.94 & 13 \\
$\{0,1,5\}$ & 0.30 & 2.13 & 5 & 0.08 & 0.41 & 1.79 & 3 \\
$\mathbf{\{0,1,6\}}$ & 1.37 & 2.09 & 7 & 0.35 & 1.85 & 8.13 & 5 \\
$\{0,1,7\}$ & 0.30 & 2.02 & 2 & 0.07 & 0.39 & 1.70 & 1 \\
$\mathbf{\{0,1,8\}}$ & 1.19 & 2.00 & 6 & 0.29 & 1.53 & 6.73 & 4 \\
$\{0,1,9\}$ & 0.59 & 1.93 & 1 & 0.14 & 0.74 & 3.25 & 1 \\
$\mathbf{\{0,2,4\}}$ & 1.04 & 1.94 & 9 & 0.25 & 1.31 & 5.75 & 2 \\
$\mathbf{\{0,2,5\}}$ & 2.96 & 2.05 & 31 & 0.74 & 3.93 & 17.28 & 15 \\
$\{0,2,9\}$ & 0.89 & 2.62 & 9 & 0.28 & 1.50 & 6.61 & 5 \\
$\{0,3,4\}$ & 0.93 & 1.84 & 6 & 0.21 & 1.11 & 4.86 & 2 \\
$\{0,3,6\}$ & 0.46 & 2.12 & 2 & 0.12 & 0.63 & 2.77 & 2 \\
$\{0,3,9\}$ & 0.30 & 2.75 & 1 & 0.10 & 0.53 & 2.32 & 0 \\
$\{0,4,5\}$ & 0.92 & 1.70 & 10 & 0.19 & 1.01 & 4.43 & 5 \\
$\mathbf{\{0,4,6\}}$ & 1.62 & 1.86 & 13 & 0.37 & 1.95 & 8.58 & 9 \\
$\{0,4,7\}$ & 0.90 & 1.99 & 2 & 0.22 & 1.16 & 5.10 & 1 \\
$\mathbf{\{0,4,8\}}$ & 1.78 & 2.16 & 9 & 0.47 & 2.48 & 10.93 & 5 \\
$\mathbf{\{0,4,9\}}$ & 2.05 & 2.74 & 8 & 0.68 & 3.63 & 15.96 & 6 \\
$\{0,5,6\}$ & 0.38 & 1.57 & 1 & 0.07 & 0.39 & 1.69 & 1 \\
$\{0,5,9\}$ & 0.30 & 2.62 & 2 & 0.09 & 0.50 & 2.20 & 2 \\
$\{0,6,7\}$ & 0.44 & 1.47 & 4 & 0.08 & 0.42 & 1.84 & 1 \\
$\{0,6,9\}$ & 1.36 & 2.44 & 4 & 0.40 & 2.15 & 9.44 & 1 \\
$\{0,7,9\}$ & 0.30 & 2.19 & 1 & 0.08 & 0.42 & 1.85 & 1 \\
$\{0,8,9\}$ & 1.19 & 1.50 & 2 & 0.22 & 1.15 & 5.03 & 1 \\
$\{1,2,3\}$ & 0.59 & 2.43 & 12 & 0.17 & 0.93 & 4.09 & 5 \\
$\mathbf{\{1,2,4\}}$ & 1.87 & 2.84 & 19 & 0.65 & 3.44 & 15.13 & 8 \\
$\mathbf{\{1,2,5\}}$ & 2.07 & 3.00 & 29 & 0.76 & 4.02 & 17.65 & 17 \\
$\{1,2,6\}$ & 1.03 & 3.00 & 12 & 0.38 & 2.00 & 8.80 & 3 \\
$\{1,2,7\}$ & 0.59 & 3.00 & 4 & 0.22 & 1.15 & 5.06 & 2 \\
$\mathbf{\{1,2,8\}}$ & 0.89 & 2.92 & 11 & 0.32 & 1.68 & 7.38 & 6 \\
$\{1,2,9\}$ & 1.19 & 2.91 & 12 & 0.42 & 2.23 & 9.80 & 4 \\
$\{1,3,4\}$ & 0.79 & 2.77 & 9 & 0.27 & 1.41 & 6.20 & 4 \\
$\{1,3,6\}$ & 1.54 & 3.28 & 18 & 0.61 & 3.26 & 14.32 & 6 \\
$\{1,3,8\}$ & 0.30 & 3.32 & 2 & 0.12 & 0.64 & 2.80 & 1 \\
$\{1,4,5\}$ & 0.57 & 2.85 & 8 & 0.20 & 1.05 & 4.61 & 4 \\
$\{1,4,6\}$ & 2.46 & 3.18 & 38 & 0.95 & 5.05 & 22.22 & 21 \\
$\{1,4,7\}$ & 0.78 & 3.34 & 8 & 0.32 & 1.68 & 7.38 & 5 \\
$\{1,4,8\}$ & 2.02 & 3.36 & 21 & 0.83 & 4.39 & 19.30 & 13 \\
$\{1,4,9\}$ & 0.91 & 3.55 & 9 & 0.39 & 2.08 & 9.15 & 3 \\
$\mathbf{\{1,5,6\}}$ & 1.47 & 2.74 & 22 & 0.49 & 2.59 & 11.40 & 15 \\
$\{1,5,8\}$ & 0.30 & 3.23 & 4 & 0.12 & 0.62 & 2.72 & 3 \\
$\{1,6,7\}$ & 1.45 & 2.64 & 8 & 0.47 & 2.48 & 10.89 & 1 \\
$\{1,6,8\}$ & 1.37 & 3.00 & 21 & 0.50 & 2.65 & 11.66 & 9 \\
$\{1,6,9\}$ & 2.06 & 3.35 & 21 & 0.84 & 4.46 & 19.62 & 11 \\
$\{1,7,8\}$ & 0.30 & 2.53 & 3 & 0.09 & 0.48 & 2.13 & 0 \\
$\{1,8,9\}$ & 0.59 & 2.49 & 8 & 0.18 & 0.95 & 4.20 & 3 \\
$\{2,3,4\}$ & 0.86 & 2.07 & 8 & 0.22 & 1.15 & 5.05 & 2 \\
$\{2,3,5\}$ & 1.48 & 2.41 & 12 & 0.43 & 2.31 & 10.14 & 2 \\
$\{2,3,6\}$ & 0.40 & 2.62 & 3 & 0.13 & 0.67 & 2.94 & 3 \\
$\{2,3,9\}$ & 0.59 & 2.66 & 11 & 0.19 & 1.02 & 4.47 & 5 \\
$\mathbf{\{2,4,5\}}$ & 2.44 & 2.35 & 18 & 0.70 & 3.70 & 16.28 & 13 \\
$\mathbf{\{2,4,6\}}$ & 1.43 & 2.63 & 13 & 0.46 & 2.43 & 10.69 & 6 \\
$\{2,4,7\}$ & 0.79 & 2.83 & 2 & 0.27 & 1.44 & 6.35 & 2 \\
$\{2,4,8\}$ & 1.05 & 2.85 & 11 & 0.36 & 1.93 & 8.50 & 7 \\
$\{2,4,9\}$ & 1.87 & 3.05 & 24 & 0.70 & 3.70 & 16.25 & 11 \\
$\{2,5,6\}$ & 4.29 & 2.52 & 35 & 1.31 & 6.98 & 30.68 & 16 \\
$\mathbf{\{2,5,7\}}$ & 1.48 & 2.80 & 12 & 0.50 & 2.68 & 11.79 & 4 \\
$\{2,5,8\}$ & 1.48 & 3.02 & 15 & 0.54 & 2.90 & 12.73 & 10 \\
$\{2,5,9\}$ & 2.07 & 3.28 & 17 & 0.83 & 4.40 & 19.33 & 5 \\
$\{2,6,7\}$ & 0.53 & 2.45 & 4 & 0.16 & 0.85 & 3.72 & 3 \\
$\{2,6,9\}$ & 0.91 & 3.13 & 13 & 0.35 & 1.84 & 8.09 & 8 \\
$\{2,7,9\}$ & 0.59 & 2.86 & 6 & 0.21 & 1.10 & 4.82 & 3 \\
$\{2,8,9\}$ & 0.89 & 2.44 & 6 & 0.26 & 1.40 & 6.17 & 3 \\
$\{3,4,5\}$ & 0.48 & 1.93 & 11 & 0.11 & 0.60 & 2.62 & 5 \\
$\{3,4,6\}$ & 1.22 & 2.20 & 12 & 0.33 & 1.74 & 7.64 & 5 \\
$\{3,4,7\}$ & 0.69 & 2.44 & 5 & 0.21 & 1.10 & 4.82 & 4 \\
$\{3,4,8\}$ & 0.91 & 2.45 & 7 & 0.27 & 1.44 & 6.35 & 5 \\
$\{3,4,9\}$ & 0.76 & 2.54 & 4 & 0.23 & 1.24 & 5.47 & 1 \\
$\{3,5,6\}$ & 0.94 & 2.11 & 14 & 0.24 & 1.29 & 5.67 & 4 \\
$\{3,6,7\}$ & 0.98 & 2.29 & 5 & 0.27 & 1.45 & 6.37 & 2 \\
$\{3,6,8\}$ & 0.44 & 2.54 & 4 & 0.14 & 0.72 & 3.18 & 3 \\
$\{3,6,9\}$ & 1.61 & 2.85 & 13 & 0.56 & 2.96 & 13.00 & 6 \\
$\{3,8,9\}$ & 0.30 & 2.22 & 4 & 0.08 & 0.42 & 1.86 & 2 \\
$\{4,5,6\}$ & 1.41 & 1.77 & 21 & 0.30 & 1.61 & 7.08 & 11 \\
$\{4,5,7\}$ & 0.46 & 2.02 & 8 & 0.11 & 0.60 & 2.63 & 3 \\
$\{4,5,8\}$ & 0.92 & 2.24 & 6 & 0.25 & 1.33 & 5.83 & 2 \\
$\{4,5,9\}$ & 0.61 & 2.34 & 6 & 0.17 & 0.92 & 4.06 & 2 \\
$\{4,6,7\}$ & 1.08 & 1.96 & 7 & 0.26 & 1.37 & 6.02 & 6 \\
$\{4,6,8\}$ & 1.63 & 2.19 & 11 & 0.44 & 2.31 & 10.17 & 4 \\
$\{4,6,9\}$ & 2.39 & 2.56 & 16 & 0.74 & 3.94 & 17.34 & 9 \\
$\{4,7,8\}$ & 0.91 & 2.12 & 4 & 0.23 & 1.24 & 5.45 & 2 \\
$\{4,7,9\}$ & 0.74 & 2.50 & 3 & 0.23 & 1.20 & 5.28 & 1 \\
$\mathbf{\{4,8,9\}}$ & 2.11 & 2.18 & 9 & 0.56 & 2.98 & 13.09 & 1 \\
$\{5,6,7\}$ & 0.78 & 1.61 & 7 & 0.15 & 0.81 & 3.56 & 6 \\
$\{5,6,8\}$ & 0.33 & 1.84 & 1 & 0.07 & 0.39 & 1.71 & 0 \\
$\{5,6,9\}$ & 1.46 & 2.07 & 9 & 0.37 & 1.96 & 8.61 & 3 \\
$\{5,8,9\}$ & 0.30 & 1.97 & 2 & 0.07 & 0.38 & 1.66 & 1 \\
$\{6,7,8\}$ & 0.44 & 1.49 & 0 & 0.08 & 0.42 & 1.87 & 0 \\
$\{6,7,9\}$ & 1.54 & 1.73 & 11 & 0.32 & 1.73 & 7.58 & 7 \\
$\{6,8,9\}$ & 1.41 & 1.70 & 12 & 0.29 & 1.55 & 6.80 & 4 \\
$\{7,8,9\}$ & 0.30 & 1.41 & 1 & 0.05 & 0.27 & 1.19 & 1 \\
\hline
    \end{tabular}}
    \caption{Estimations for the number of sporadic solutions.}
    \label{table_est_sporadic_solutions}
  \end{center}
\end{table}

Let us consider the following heuristic argument. For a given base $B\ge 2$ and a subset $D\subseteq\{0,1,\dots,B-1\}$ of allowed digits we assume that there exist 
constants $c_1,c_2\in\mathbb{R}_{\ge 0}$ such that the number of $k$-suffixes is roughly $c_1\cdot |D|^k$ and the number of $m$-prefixes is roughly $c_2\cdot |D|^m$. 
Combining an $m$-prefix with a $k$-suffix gives an $n$-digit root $r$, where $n=m+k$. By construction the last $k$ digits of $r^2$ are contained in $D$. If the 
carryovers don't interfere, then also the first $m$ digits of $r^2$ are contained in $D$. Noting that $r^2$ contains either $2n-1$ or $2n$ digits, the $n-1$ or $n$ 
digits in the center have to be checked. For each digit we assume an independent probability of $|D|/B$. Under this assumptions the expectation for an 
$n$-digit root $r$ with $\varpi_B(r^2)=D$ is given by
\begin{equation}
  c_1c_2|D|^n \cdot \left(\frac{|D|}{B}\right)^{n-s}= c_1c_2\cdot\frac{|D|^{2n-s}}{B^{n-s}}, 
\end{equation}  
where $s\in \{0,1\}$. For simplicity we assume $s=0$, so that the expectation is roughy $c_1c_2\left(\frac{|D|^2}{B}\right)^n$. 
By summing we conclude that the expected number of root solutions with at least $n$ digits is given by 
\begin{equation}
  c_1c_2\left(\frac{|D|^2}{B}\right)^n\cdot \frac{B}{B-|D|}.
\end{equation}
For $B=10$ and $|D|=3$ this expectation tends to zero for increasing $n$, so that the existence of infinite patterns shows that our independence assumption is wrong 
in general. Nevertheless the corresponding estimations are not too bad in many cases as shown exemplarily in Table~\ref{table_est_sporadic_solutions}. Here we have chosen 
the values for $c_1$ and $c_2$ according to the number of $10$-suffixes and $10$-prefixes. We give counts for the number of sporadic root solutions of length $25$ or 
length between $10$ and $25$. Estimations are stated for the cases of $20$-, $>25$, and $10-25$ digits. Removing the solutions from infinite patterns certainly has 
an effect, so that we mark the triples $D$ of allowed digits that permit infinite patterns in bold face. Triples $D$ where $c_1c_2$ is relatively large allow infinite patterns 
in many cases. A notably exception is $D=\{1,4,6\}$, which seems to allow no infinite pattern but many sporadic solutions. Here we have $c_1c_2=7.82$, the estimated number of 
solutions with $10-25$ digits is $22.22$ and there are indeed $21$ of them. The estimation for the number of solutions with more than $25$ digits is even as large as $5.05$, 
so that searching more extensively seems to make sense.
 

\appendix
\renewcommand{\thesubsection}{\Alph{subsection}}
\section*{Appendix}
\subsection{Results for base $\mathbf{B=9}$}
\label{sec_b_9}

In Table~\ref{table_explanations_b_9_patterns} we list a few infinite patterns in base $B=9$. Since $B=3^2$ we have 
$\varphi_9((3N)^2)=\varphi_9(N^2)\cup\{0\}$. So, for the cases $0\in D$ of allowed digits we assume that $3$ does not divide $N$. 
A corresponding {\lq\lq}dual pattern{\rq\rq} always exists and can be constructed easily. For applications of 
Theorem~\ref{thm_square_digits_and_double_square_digits_everywhere} with $u=1$ we only give three examples. Due to the relation between 
examples with decimal digits in $\{0,1,2\}$ and those with decimal digits in $\{0,4,8\}$, see Lemma~\ref{lemma_multiplied_triples}, we 
have dropped the patterns for the latter. Patterns constructed by Lemma~\ref{lemma_r_s_construction} come in pairs if $r\neq s$ and we 
only state the examples with $r\le s$ or $r>s$ if $s$ would be divisible by $3$ otherwise.  

\begin{table}[htp]
  \begin{center}
    \begin{tabular}{cccc}
      \hline
      pattern & condition & $D$ & explanation \\ 
      \hline
      $1(0)_n1$ & $n\ge 0$ & $\{0,1,2\}$ & Theorem~\ref{thm_square_digits_and_double_square_digits_everywhere} with $u=1$\\
      $1(0)_n1(0)_{n+1}1$ & $n\ge 0$ & $\{0,1,2\}$ & Theorem~\ref{thm_square_digits_and_double_square_digits_everywhere} with $u=1$\\
      $1(0)_n11$ & $n\ge 1$ & $\{0,1,2\}$ & Theorem~\ref{thm_square_digits_and_double_square_digits_everywhere} with $u=1$\\
      $1(0)_n1(0)_n4(0)_n1(0)_n1$ & $n\ge 1$ & $\{0,1,2\}$ & Lemma~\ref{lemma_ones_with_4_in_center_b9}\\   
      $1(0)_n5(0)_n1$ & $n\ge 1$ & $\{0,1,3\}$ & Lemma~\ref{lemma_two_blocks_of_zeros} with $\left(r_0,r_1,r_2;c\right)=(1,5,1;0)$\\
      $1(0)_n16(0)_n1$ & $n\ge 1$ & $\{0,1,3\}$ & Lemma~\ref{lemma_two_blocks_of_zeros} with $\left(r_0,r_1,r_2;c\right)=(1,15,1;1)$\\
      $1(0)_n2$ & $n\ge 0$ & $\{0,1,4\}$ & Lemma~\ref{lemma_r_s_construction} with $r=1$, $s=2$\\
      $1(0)_n2(0)_k2(0)_n1$ & $n>k\ge 1$ & $\{0,1,4\}$ & Lemma~\ref{lemma_1_2_construction_for_base_9_gen}\\
      $2(0)_n1(0)_k1(0)_n2$ & $n>k\ge 1$ & $\{0,1,4\}$ & Lemma~\ref{lemma_1_2_construction_for_base_9_gen}\\
      $1(0)_n 20222(0)_{n+3}1$ & $n\ge 1$ & $\{0,1,4\}$ & Lemma~\ref{lemma_two_blocks_of_zeros} with $\left(r_0,r_1,r_2;c\right)=(1,13304,1;1)$\\
      $21(0)_n85(0)_{n}21$ & $n\ge 1$ & $\{0,1,4\}$ & Lemma~\ref{lemma_two_blocks_of_zeros} with $\left(r_0,r_1,r_2;c\right)=(19,77,19;1)$\\      
      $3(0)_n1$ & $n\ge 0$ & $\{0,1,6\}$ & Lemma~\ref{lemma_r_s_construction} with $r=3$, $s=1$\\      
      $126(0)_n38$ & $n\ge 2$ & $\{0,1,6\}$ & Lemma~\ref{lemma_r_s_construction} with $r=105$, $s=35$\\      
      $1(0)_n8$ & $n\ge 0$ & $\{0,1,7\}$ & Lemma~\ref{lemma_r_s_construction} with $r=1$, $s=8$\\      
      $4(0)_n8$ & $n\ge 0$ & $\{0,1,7\}$ & Lemma~\ref{lemma_r_s_construction} with $r=4$, $s=8$\\
      $8(0)_n10434$ & $n\ge 2$ & $\{0,1,7\}$ & Lemma~\ref{lemma_r_s_construction} with $r=8$, $s=6916$\\
      $8(0)_n4(0)_n8$ & $n\ge 1$ & $\{0,1,7\}$ & Lemma~\ref{lemma_two_blocks_of_zeros} with $\left(r_0,r_1,r_2;c\right)=(8,4,8;0)$\\      
      $8(0)_n10434(0)_{n+2}8$ & $n\ge 2$ & $\{0,1,7\}$ & Lemma~\ref{lemma_two_blocks_of_zeros} with $\left(r_0,r_1,r_2;c\right)=(8,6916,8;2)$\\      
      $1(0)_{n+1}8(0)_n1$ & $n\ge 1$ & $\{0,1,7\}$ & Lemma~\ref{lemma_two_blocks_of_zeros} with $\left(r_0,r_1,r_2;c\right)=(1,8,1;0)$\\      
      $8(0)_n104(0)_{n+2}8$ & $n\ge 2$ & $\{0,1,7\}$ & Lemma~\ref{lemma_two_blocks_of_zeros} with $\left(r_0,r_1,r_2;c\right)=(8,85,8;0)$\\
      $(8)_n5(0)_n1$ & $n\ge 1$ & $\{0,1,8\}$ & Lemma~\ref{lemma_b9_infinite_1paramt_1}\\      
      $1 (0)_m(8)_k5(0)_k1$ & $2\le k+1\le m<2k$ & $\{0,1,8\}$ & Lemma~\ref{lemma_b9_infinite_2paramt}\\     
      $(4)_n5$ & $n\ge 1$ & $\{2,6,7\}$ & Section~\ref{subsec_infinite_non_zero}\\
      \hline
    \end{tabular}
    \caption{Explanations for some infinite patterns in base $B=9$.}
    \label{table_explanations_b_9_patterns}  
  \end{center}
\end{table}

\begin{lemma}
  \label{lemma_1_2_construction_for_base_9_gen}
  For $B=9$ let
  $$
    N_1=\big[1,\underset{k\text{ times}}{\underbrace{0,\dots,0}},2,\underset{n\text{ times}}{\underbrace{0,\dots,0}},2,\underset{k\text{ times}}{\underbrace{0,\dots,0}},1\big]_{B} 
    =1+2\cdot B^{k+1}+2\cdot B^{n+k+2}+1\cdot B^{n+2k+3}
  $$
  and
  $$
    N_2=\big[2,\underset{k\text{ times}}{\underbrace{0,\dots,0}},1,\underset{n\text{ times}}{\underbrace{0,\dots,0}},1,\underset{k\text{ times}}{\underbrace{0,\dots,0}},2\big]_{B} 
    =2+1\cdot B^{k+1}+1\cdot B^{n+k+2}+2\cdot B^{n+2k+3}.
  $$
  Then, $N_1^2$ and $N_2^2$ have digits in $\{0,1,4\}$ in base $B$ if 
  $$
    k+n+2,\,\, 2k+2,\,\, 2k+n+3,\,\, 2k+n+4,\,\, 2k+2n+4,\text{ and }3k+n+4
  $$
  are pairwise different. The conditions are e.g.\ satisfied if $k>n$.
\end{lemma}
\begin{proof}
  We have
  $$
    N_1^2= 1 + 4\cdot B^{k+1} + 4\cdot B^{k+n+2}  +4\cdot B^{2k+2} + 10\cdot B^{2k+n+ 3}  +  4\cdot B^{2k+2n+4} + 4\cdot B^{3k+n+4}  +4\cdot B^{3k+2n+5}  + 1\cdot B^{4k+2n+6}
  $$
  and
  $$
    N_2^2=4 + 4\cdot B^{k+1} + 4\cdot B^{k+n+2} + 1\cdot B^{2k+2} + 10\cdot B^{2k+n+3}  + 1\cdot B^{2k+2n+4} + 4\cdot B^{3k+n+4}  + 4\cdot B^{3k+2n+5} + 4\cdot B^{4k+2n+6}.
  $$
  In base $B=9$ we can replace the terms $10\cdot B^{2k+n+ 3}$ by $1\cdot B^{2k+n+ 3}+1\cdot B^{2k+n+ 4}$. Note that the pairs of exponents that are not mentioned in the 
  statement cannot coincide.
\end{proof}

\begin{lemma}
  \label{lemma_ones_with_4_in_center_b9}
  For each $n\ge 2$ and $B=9$ the integer
  \begin{equation}
    N=1+1\cdot B^n+4\cdot B^{2n}+1\cdot B^{3n}+1\cdot B^{4n}
  \end{equation}
  satisfies $D_B(N^2)=\{0,1,2\}$.
\end{lemma}  
\begin{proof}
  We compute 
  $$ 
    N^2= 1+ 2\cdot B^n + 9\cdot B^{2n} + 10\cdot B^{3n} + 20\cdot B^{4n} + 10\cdot B^{5n} + 9\cdot B^{6n} + 2\cdot B^{7n} + B^{8n},
  $$  
  so that $9=1\cdot 9+0$, $10=1\cdot 9+1$, and $20=2\cdot 9+2$ imply the stated result.
\end{proof}

\begin{lemma}
  \label{lemma_b9_infinite_1paramt_1}
  For $B=9$ let
  \begin{equation}
    N=1 + 5\cdot B^k + (B^{k-1}-1)\cdot B^{k+1}.
  \end{equation} 
  If $k\ge 2$, then $\varphi_9(N^2)=\{0,1,B-1\}=\{0,1,8\}$.
\end{lemma}
\begin{proof}
  We compute
  \begin{eqnarray*}
  N^2 &=& 1+ 10\cdot B^k- 2\cdot B^{k + 1} + 27\cdot B^{2k}- 10\cdot B^{2k + 1} + B^{2k + 2} + 10\cdot B^{3k}- 2\cdot B^{3k + 1} + B^{4k}\\ 
  &\overset{B=9}{=}& 1+ 1\cdot B^k- 1\cdot B^{k + 1} + 2\cdot B^{2k+1}  + 1\cdot B^{3k}- 1\cdot B^{3k + 1} + B^{4k}\\ 
  &=& 1+ 1\cdot B^k +  B^{k + 1}\cdot \left(B^k-1\right) + 1\cdot B^{2k+1} + 1\cdot B^{3k} + B^{3k+1}\cdot\left(B^{k-1}-1\right). 
  \end{eqnarray*}
  The sequences $B^x\cdot\left(B^y-1\right)$ correspond to a sequence of $y$ times digit $B-1=8$. Note that for $k\ge 2$ the indicated 
  digits $1$ and $8$ do not interfere. 
\end{proof}  

\begin{lemma}
  \label{lemma_b9_infinite_2paramt}
  For $B=9$ let
  \begin{equation}
    N=1 + 5\cdot B^k + (B^{k-1}-1)\cdot B^{k+1} + B^m.
  \end{equation} 
  If $k\ge 2$ and $3k+1\le m<4k$, then $\varphi_9(N^2)=\{0,1,B-1\}=\{0,1,8\}$.
\end{lemma}
\begin{proof}
  We compute
  \begin{eqnarray*}
  N^2 &=& 1+ 10\cdot B^k- 2\cdot B^{k + 1} + 27\cdot B^{2k}- 10\cdot B^{2k + 1} + B^{2k + 2} + 10\cdot B^{3k}- 2\cdot B^{3k + 1} + 2\cdot B^m + B^{4k}\\ 
  &&  + 10\cdot B^{m+k}- 2\cdot B^{m+k+1} + 2\cdot B^{m+2k} + B^{2m}\\ 
  &\overset{B=9}{=}& 1+ 1\cdot B^k- 1\cdot B^{k + 1} + 2\cdot B^{2k+1}  + 1\cdot B^{3k}- 1\cdot B^{3k + 1} + 2\cdot B^m + B^{4k}\\ 
  &&  + 1\cdot B^{m+k}- 1\cdot B^{m+k+1} + 2\cdot B^{m+2k} + B^{2m}\\
  &=& 1+ 1\cdot B^k +  B^{k + 1}\cdot \left(B^k-1\right) + 1\cdot B^{2k+1} + 1\cdot B^{3k} + B^{3k+1}\cdot\left(B^{m-3k-1}-1\right)+ 1\cdot B^m+1\cdot B^{4k} \\ 
  && + 1\cdot B^{m+k} +B^{m+k+1}\cdot\left(B^{k-1}-1\right)  + 1\cdot B^{m+2k} + 1\cdot B^{2m}.  
  \end{eqnarray*}
  The sequences $B^x\cdot\left(B^y-1\right)$ correspond to a sequence of $y$ times digit $B-1=8$. Note that $k-1,m-3k-1\ge 0$ and $m<4k<m+k$, so that the indicated 
  digits $1$ and $8$ do not interfere. 
\end{proof}  

\subsection{Squares with just two digits in base $\mathbf{B=10}$}
In this section we want to briefly summarize the known results on positive integers $N$ such are not divisible by $10$ but satisfy $|\varpi_{10}(N^2)|=2$. The only known 
examples are given by 
$$
  4, 5, 6, 7, 8, 9, 11, 12, 15, 21, 22, 26, 38, 88, 109, 173, 212, 235, 264, 3114, 81619,
$$
see e.g.\ \url{https://oeis.org/A016070} and \cite[Entry 109]{de2009those}. The corresponding squares are given by 
{\footnotesize$$
  16, 25, 36, 49, 64, 81, 121, 144, 225, 441, 484, 676, 1444, 7744, 11881, 29929, 44944, 55225, 69696, 9696996, 6661661161.
$$}  
Below $10^{41}$ every other square either ends with zero or contains at least three different digits. Given the heuristic argument in Section~\ref{sec_search_algorithms}, 
it is very unlikely that further examples exist.

There are exactly $22$ possibilities for the last two decimal digits:
$$
  01, 04, 09, 12, 14, 16, 18, 24, 25, 29, 34, 36, 33, 45, 46, 47, 48, 49, 56, 67, 69, 89,
$$  	
i.e., the last two digits already determine $\varphi_{10}(N^2)$ except when the square ends with $44$. For some construction principles for infinite patterns, as described in Section~\ref{sec_infinite_patterns}, 
it can be easily shown that no such example exists for just two different decimal digits. With respect to Lemma~\ref{lemma_r_s_construction} we remark 
that the parameters $r$ and $s$ itself have to satisfy $\left|\varphi_{10}(r^2)\backslash \{0\}\right|=\left|\varphi_{10}(s^2)\backslash \{0\}\right|=1$, which just leaves the cases 
$r,s\in\{1,2,3\}$.
  	
\end{document}